\documentclass[a4paper]{article}
\usepackage{amssymb,amsmath,amsthm,eucal}
\makeatletter
 
  \@addtoreset{equation}{section}
 \makeatother
\theoremstyle{plain}

\newtheorem{theorem}{Theorem}[section]
\newtheorem{lemma}[theorem]{Lemma}
\newtheorem{proposition}[theorem]{Proposition}

\theoremstyle{definition}

\newtheorem{remark}[theorem]{Remark}

\newtheorem{assumption}{Assumption}

\newenvironment{bew}[2]{\removelastskip\vspace{6pt}\noindent
 {\it Proof  #1.}~\rm#2}{\par\vspace{6pt}}

\newcommand{\norm}[1]{{||#1||}}
\newcommand{\abs}[1]{{\left|#1\right|}}

\newcommand{\dderiv}[4]{{\partial_{#1}^{#3}\partial_{#2}^{#4}}}

\newcommand{\wtilde}[1]{{\widetilde{#1}}}

\def\supp{\mathop{\mathrm{supp}}\nolimits}

\def\Id{\mathop{\mathrm{Id}}\nolimits}

\def\IK{\mathop{\mathrm{IK}}\nolimits}
\def\WKB{\mathop{\mathrm{WKB}}\nolimits}

\def\R{{\mathbb{R}}}
\def\Z{{\mathbb{Z}}}

\def\Sphere{{\mathbb{S}}}
\def\S{{\mathcal{S}}}
\def\F{{\mathcal{F}}}

\def\B{{\mathcal{B}}}
\def\L{{\mathcal{L}}}
\def\X{{\mathcal{X}}}
\def\Y{{\mathcal{Y}}}

\def\<{{\langle}}
\def\>{{\rangle}}
\def\ep{{\varepsilon}}

\title
{Strichartz estimates for Schr\"odinger equations with variable coefficients and potentials at most linear at spatial infinity}

\author{Haruya Mizutani${}^*$}
\date{\empty}

\begin{document}

\maketitle

\footnotetext{${}^*$Research Institute for Mathematical Sciences, Kyoto University, Kyoto 606-8502, Japan. E-mail: \texttt{hmizutan@kurims.kyoto-u.ac.jp}. Partly supported by GCOE `Fostering top leaders in mathematics', Kyoto University. 
}

\begin{abstract}
In the present paper we consider Schr\"odinger equations with variable coefficients and potentials, where the principal part is a long-range perturbation of the flat Laplacian and potentials have at most linear growth at spatial infinity. We then prove local-in-time Strichartz estimates, outside a large compact set centered at origin, expect for the endpoint. Moreover we also prove global-in-space Strichartz estimates under the non-trapping condition on the Hamilton flow generated by the kinetic energy. \end{abstract}

\section{Introduction}
In this paper we study the so called (local-in-time) \emph{Strichartz estimates} for the solutions to $d$-dimensional time-dependent Schr\"odinger equations
\begin{equation}
\begin{aligned}
\label{Sch_equation}
i\partial_t u(t)=Hu(t),\ t \in \R;\quad u|_{t=0}=u_0 \in L^2(\R^d),
\end{aligned}
\end{equation}
where $d \ge 1$ and $H$ is a Schr\"odinger operator with variable coefficients:
$$
H=-\frac12 \sum_{j,k=1}^d \partial_{x_j} a^{jk}(x)\partial_{x_k} +V(x).
$$
 Throughout the paper we assume that $a^{jk}(x)$ and $V(x)$ are real-valued and smooth on $\R^d$, and $(a^{jk}(x))$ is a symmetric matrix satisfying
$$
C^{-1}\Id \le (a^{jk}(x)) \le C\Id,\quad x \in \R^d,
$$
with some $C>0$. We also assume 
\begin{assumption}
\label{assumption_A}
There exist constants $\mu,\nu \ge 0$ such that, for any $\alpha \in \Z^d_+$, 
\begin{align*}
\abs{\partial_x^\alpha (a^{jk}(x)-\delta_{jk})}\le C_\alpha \<x\>^{-\mu-|\alpha|},\quad 
\abs{\partial_x^\alpha V(x)} \le C_\alpha \<x\>^{2-\nu-|\alpha|},\quad x \in \R^d,
\end{align*}
with some $C_\alpha>0$. 
\end{assumption}
We may assume $\mu<1$ and $\nu<2$ without loss of generality. 
It is well known that $H$ is essentially self-adjoint on $C_0^\infty(\R^d)$ under Assumption \ref{assumption_A}, and we denote the unique self-adjoint extension on $L^2(\R^d)$  by the same symbol $H$. 
By the Stone theorem, the solution to \eqref{Sch_equation} is given by $u(t)=e^{-itH}u_0$, where $e^{-itH}$ is a unique unitary group on $L^2(\R^d)$ generated by $H$ and called the propagator. 

Let us recall the (global-in-time) Strichartz estimates for the free Schr\"odinger equation state that
\begin{align}
\label{Strichartz_estimates_1}
\norm{e^{it\Delta/2}u_0}_{L^p(\R;L^q(\R^d)} \le C \norm{u_0}_{L^2(\R^d)},
\end{align}
where $(p,q)$ satisfies the following \emph{admissible} condition
\begin{align}
\label{admissible}
2\le p,q \le \infty,\quad \frac2p+\frac{d}{q}=\frac{d}{2},\quad  (d,p,q) \neq (2,2,\infty).
\end{align}
For $d \ge 3$, $(p,q)=(2,\frac{2d}{d-2})$ is called the endpoint. It is well known that these estimates are fundamental in studying the local well-posedness of Cauchy problem of nonlinear Schr\"odinger equations (see, \emph{e.g.,} \cite{Cazenave}). 
The estimates \eqref{Strichartz_estimates_1} were first proved by Strichartz \cite{Strichartz} for a restricted pair of $(p,q)$ with $p=q=2(d+2)/d$, and have been extensively generalized for $(p,q)$ satisfying \eqref{admissible} by \cite{Ginibre_Velo,Keel_Tao}. Moreover, in the flat case ($a^{jk}\equiv\delta_{jk}$), local-in-time Strichartz estimates
\begin{align}
\label{Strichartz_estimates_2}
\norm{e^{itH}u_0}_{L^p([-T,T];L^q(\R^d))} \le C_T \norm{u_0}_{L^2(\R^d)},
\end{align}
have been extended to the case with potentials decaying at infinity \cite{Yajima1} or increasing at most quadratically at infinity \cite{Yajima2}. In particular, if $V(x)$ has at most quadratic growth at spatial infinity, \emph{i.e.}, 
$$V \in C^\infty(\R^d;\R),\quad |\partial_x^\alpha V(x)| \le C_\alpha\ \text{for}\ |\alpha| \ge 2,$$ then it was shown by Fujiwara \cite{Fujiwara} that the fundamental solution $E(t,x,y)$ of the propagator $e^{-itH}$ satisfies 
$$
|E(t,x,y)| \lesssim |t|^{-d/2},\quad x,y \in \R^d,
$$ 
for $t \neq 0$ small enough. The estimates \eqref{Strichartz_estimates_2} are immediate consequences of this estimate and the $TT^*$-argument due to Ginibre-Velo \cite{Ginibre_Velo} (see Keel-Tao \cite{Keel_Tao} for the endpoint estimate). For the case with magnetic fields or singular potentials, we refer to Yajima \cite{Yajima2,Yajima3} and references therein. 

On the other hand, local-in-time Strichartz estimates on manifolds have recently been proved by many authors under several conditions on the geometry. 
Staffilani-Tataru \cite{Staffilani_Tataru}, Robbiano-Zuily \cite{Robbiano_Zuily} and Bouclet-Tzvetkov \cite{Bouclet_Tzvetkov_1} studied the case on the Euclidean space with the asymptotically flat metric under several settings. 
In particular, Bouclet-Tzvetkov \cite{Bouclet_Tzvetkov_1} proved local-in-time Strichartz estimates without loss of derivatives under Assumption A with $\mu>0$ and $\nu>2$ and the non-trapping condition. 
 Burq-G\'erard-Tzvetkov \cite{BGT} proved Strichartz estimates with  a loss of derivative $1/p$ on any compact manifolds without boundaries. 
They also proved that the loss $1/p$ is optimal in the case of $M=\Sphere^d$. Hassell-Tao-Wunsch \cite{HTW} and the author \cite{Mizutani} considered the case of non-trapping asymptotically conic manifolds which are non-compact Riemannian manifolds with an asymptotically conic structure at infinity. 
Bouclet \cite{Bouclet} studied the case of an asymptotically hyperbolic manifold. 
Burq-Guillarmou-Hassell \cite{BGH} recently studied the case of asymptotically conic manifolds with hyperbolic trapped trajectories of sufficiently small fractal dimension. For global-in-time Strichartz estimates, we refer to \cite{EGS,D'AFVV} and the references therein in the case with electromagnetic potentials, and to \cite{Bouclet_Tzvetkov_2,Tataru,MMT} in the case of Euclidean space with an asymptotically flat metric. 

The main purpose of the paper is to handle a mixed case of above two situations. More precisely, we show that local-in-time Strichartz estimates for long-range perturbations still hold (without loss of derivatives) if we add unbounded potentials which have at most linear growth at spatial infinity (\emph{i.e.}, $\nu\ge 1$), at least excluding the endpoint $(p,q)=(2,2d/(d-2))$. To the best knowledge of the author, our result may be a first example on the case where both of variable coefficients and unbounded potentials in the spatial variable $x$ are present.

To state the result, we recall the non-trapping condition. We denote by
$$
H_0=H-V=-\frac12\sum_{j,k=1}^d\partial_{x_j}a^{jk}(x)\partial_{x_k},\quad 
k(x,\xi)=\frac12\sum_{j,k=1}^da^{jk}(x)\xi_j\xi_k,
$$
the principal part of $H$ and the kinetic energy, respectively, and also denote by $$(y_0(t,x,\xi),\eta_0(t,x,\xi))$$ the Hamilton flow generated by $k(x,\xi)$:
$$
\dot{y}_0(t)=\partial_\xi k(y_0(t),\eta_0(t)),\ \dot{\eta}_0(t)=-\partial_x k(y_0(t),\eta_0(t));\quad
(y_0(0),\eta_0(0))=(x,\xi). 
$$
Note that the Hamiltonian vector field $H_{k}$, generated by $k$, is complete on $\R^{2d}$ since $(a^{jk})$ satisfies the uniform elliptic condition, and $(y_0(t,x,\xi),\eta_0(t,x,\xi))$ hence exists for all $t \in \R$. We consider the following \emph{non-trapping condition}:
\begin{align}
\label{non-trapping}
\text{For any}\ (x,\xi) \in T^*\R^d\ \text{with}\ \xi \neq 0,\ |y_0(t,x,\xi)| \to +\infty\ as\ t \to \pm \infty. 
\end{align}

We now state our main result. 
\begin{theorem}						
\label{theorem_1}
\emph{(i)} Suppose that $H$ satisfies Assumption \ref{assumption_A} with $\mu>0$ and $\nu \ge 1$. Then, there exist  $R_0>0$ large enough and $\chi_0\in C^\infty_0(\R^d)$ with $\chi_0(x)=1$ for $|x| <R_0$ such that, for any $T>0$ and $(p,q)$ satisfying \eqref{admissible} and $p\neq2$, there exists $C_T>0$ such that
\begin{align}
\label{theorem_1_1}
\norm{(1-\chi_0)e^{-itH}u_0}_{L^p([-T,T];L^q(\R^d))} \le C_T \norm{u_0}_{L^2(\R^d)}. 
\end{align}
\emph{(ii)} Suppose that $H$ satisfies Assumption \ref{assumption_A} with $\mu=\nu=0$ and $k(x,\xi)$ satisfies the non-trapping condition \eqref{non-trapping}. Then, for any $\chi \in C^\infty_0(\R^d)$, $T>0$ and $(p,q)$ satisfying \eqref{admissible} and $p\neq2$, we have
\begin{align}
\label{theorem_1_2}
\norm{\chi e^{-itH}u_0}_{L^p([-T,T];L^q(\R^d))} \le C_T \norm{u_0}_{L^2(\R^d)}.
\end{align}
Moreover, combining with \eqref{theorem_1_1}, we obtain global-in-space estimates
$$\norm{e^{-itH}u_0}_{L^p([-T,T];L^q(\R^d))} \le C_T \norm{u_0}_{L^2(\R^d)},$$
provided that $\mu>0$ and $\nu \ge1$. 
\end{theorem}

We here display the outline of the paper and explain the idea of the proof of Theorem \ref{theorem_1}. 
By the virtue of the Littlewood-Paley theory in terms of $H_0$, the proof of \eqref{theorem_1_1} can be reduced to that of following semi-classical Strichartz estimates:
$$
\norm{(1-\chi_0)\psi(h^2H_0)e^{-itH}u_0}_{L^p([-T,T];L^q(\R^d))} \le C_T \norm{u_0}_{L^2(\R^d)},\quad 0<h\ll 1,
$$
where $\psi \in C^\infty_0(\R)$ with $\supp \psi \Subset (0,\infty)$ and $C_T>0$ is independent of $h$. Moreover, there exists a smooth function $a \in C^\infty(\R^{2d})$ supported in a neighborhood of the support of $(1-\chi_0)\psi\circ k$ such that $(1-\chi_0)\varphi(h^2H_0)$ can be replaced with semi-classical pseudodifferential operator $a(x,hD)$.
In Section \ref{Preliminaries}, we collect some known results on the semi-classical pseudo-differential calculus and prove such a reduction to semi-classical estimates.  Rescaling $t \mapsto th$, we want to show dispersive estimates for $e^{ithH}$on a time scale of order $h^{-1}$ for proving semi-classical Strichartz estimates. To prove dispersive estimates, we construct two kinds of parametrices, namely the Isozaki-Kitada and the WKB parametrices. Let $a^\pm\in S(1,dx^2/\<x\>^2+d\xi^2/\<\xi\>^2)$ be symbols supported in the following outgoing and incoming regions:
$$
\{(x,\xi);|x|>R_0,\ |\xi|^2\in J,\ \pm x\cdot\xi>-(1/2)|x||\xi|\}, 
$$
respectively, where $J\Subset (0,\infty)$ is an open interval so that $\pi_{\xi}(\supp\psi\circ k) \Subset J$ and $\pi_\xi$ is the projection onto the $\xi$-space. If $H$ is a long-range perturbation of $-(1/2)\Delta,$ then the outgoing (resp. incoming) Isozaki-Kitada parametrix of $e^{-itH}a^+(x,hD)$  for $0 \le t \le h^{-1}$ (resp. $e^{-itH}a^-(x,hD)$  for $-h^{-1} \le t \le 0$) has been constructed by Robert \cite{Robert} (see, also \cite{Bouclet_Tzvetkov_1}). However, because of the unboundedness of $V$ with respect to $x$, it is difficult to construct such parametrices of $e^{-ithH}a^\pm(x,hD)$ . To overcome this difficulty, we use a method due to Yajima-Zhang \cite{Yajima_Zhang} as follows. We approximate $e^{-ithH}$ by $e^{-ithH_h}$, where $H_h=H-V+V_h$ and $V_h$ vanishes in the region $\{x;|x|\gg h^{-1}\}$. Suppose that $a^+$ (resp. $a^-$) is supported in the intersection of the outgoing (resp. incoming) region and $\{x;|x|< h^{-1}\}$. In Section \ref{Isozaki-Kitada_parametrix}, we construct the Isozaki-Kitada parametrix of $e^{-ithH_h}a^\pm(x,hD)$ for $0 \le \pm t \le h^{-1}$ and prove the following justification of the approximation: for any $N>0$, 
$$
\sup_{0 \le \pm t \le h^{-1}}\norm{(e^{-ithH}-e^{-ithH_h})a^\pm(x,hD)f}_{L^2} \le C_Nh^N\norm{f}_{L^2},\quad 0<h\ll 1. 
$$
In Section \ref{WKB_parametrix}, we discuss the WKB parametrix construction of $e^{-ithH}a(x,hD)$ on a time scale of order $h^{-1}$, where $a$ is supported in $\{(x,\xi);|x|>h^{-1},\ |\xi|^2 \in I\}$. Such a parametrix construction is basically known for the potential perturbation case (see, \emph{e.g.}, \cite{Yajima4}) and has been proved by the author for the case on asymptotically conic manifolds \cite{Mizutani}. Combining these results studied in Sections \ref{Preliminaries}, \ref{Isozaki-Kitada_parametrix} and \ref{WKB_parametrix} with the Keel-Tao theorem \cite{Keel_Tao}, we prove semi-classical Strichartz estimates in Section \ref{Proof}. Section \ref{compact} is devoted to the proof of \eqref{theorem_1_2}. The proof heavily depends on local smoothing effects due to Doi \cite{Doi} and the Chirist-Kiselev lemma \cite{Christ_Kiselev}. The method of the proof is similar as that in Robbiano-Zuily \cite{Robbiano_Zuily}. Appendix \ref{appendix_A} is devoted to prove some technical inequalities on the Hamilton flow needed for constructing the WKB parametrix. 

Throughout the paper we use the following notations. For $A,B \ge 0$, $A \lesssim B$ means that there exists some universal constant $C>0$ such that $A \le CB$. We denote the set of multi-indices by $\Z^d_+$. For Banach spaces $X$ and $Y$, $\L(X,Y)$ denotes the Banach space of bounded operators from $X$ to $Y$, and we write $\L(X):=\L(X,X)$.


\section{Reduction to semi-classical estimates}
\label{Preliminaries}
We here show that Theorem \ref{theorem_1} (i) follows from semi-classical Strichartz estimates. We first record known results on the pseudo-differential calculus and the $L^p$-functional calculus. For any symbol $a \in C^\infty(\R^{2d})$ and $h \in (0,1]$, we denote the semi-classical pseudo-differential operator ($h$-PDO for short) by $a(x,hD_x)$:
$$
a(x,hD_x)u(x)=(2\pi h)^{-d}\int e^{i(x-y)\cdot\xi/h}a(x,\xi)u(y)dyd\xi,\quad u \in \S(\R^d), 
$$
where $\S(\R^d)$ is the Schwartz class. For a metric 
$$
g=dx^2/\<x\>^2+d\xi^2/\<\xi\>^2\ \text{on}\ T^*\R^d,
$$
we consider H\"ormander's symbol class $S(m,g)$ with a weighted function $m$, namely we write $a \in S(m,g)$ if $a \in C^\infty(\R^{2d})$ and 
$$
|\dderiv{x}{\xi}{\alpha}{\beta}a(x,\xi)| \le C_{\alpha\beta}m(x,\xi)\<x\>^{-|\alpha|}\<\xi\>^{-|\beta|}, \quad
x,\xi \in\R^d.
$$
Let $a \in S(m_1,g)$, $b \in S(m_2,g)$. For any $N = 0,1,2,...$, the symbol of the composition $a(x,hD)b(x,hD)$, denoted by $a\sharp b$, has an asymptotic expansion
\begin{align}
\label{symbolic_calculus}
a\sharp b(x,\xi)=\sum_{|\alpha| \le N}^N\frac{h^{|\alpha|}}{i^{|\alpha|}\alpha!}\partial_\xi a(x,\xi)\cdot\partial_x b(x,\xi)+h^{N+1}r_N(x,\xi)
\end{align}
with some $r_N \in S(\<x\>^{-N-1}\<\xi\>^{-N-1}m_1m_2,g)$. 
For $a \in S(1,g)$, $a(x,hD_x)$ is extended to a bounded operator on $L^2(\R^d)$. Moreover, if $a \in S(\<\xi\>^{-N},g)$ for some $N>d$, then the distribution kernel $A_h(x,y)$ of $a(x,hD)$ satisfies
$$\sup_x\int |A_h(x,y)|dy + \sup_y\int |A_h(x,y)|dx \le C$$
for some $C>0$ independent of $h$. By using this estimate, the Schur lemma and an interpolation, we have
$$\norm{a(x,hD)}_{\L(L^q(\R^d),L^r(\R^d))} \le C_{qr}h^{-d(1/q-1/r)},\quad 1\le q\le r \le \infty,\ h\in (0,1],$$
where $C_{qr}>0$ is independent of $h$. 

We next consider the $L^p$-functional calculus. The following lemma, which has been proved by \cite[Proposition 2.5]{Bouclet_Tzvetkov_1}, tells us that for any $\varphi \in C^\infty_0(\R)$ with $\supp \varphi \Subset (0,\infty)$, $\varphi(h^2H_0)$ can be approximated in terms of the $h$-PDO.

\begin{lemma}					
\label{2_lemma_1}
Let $\varphi \in C_0^\infty(\R)$, $\supp \varphi \Subset (0,\infty)$ and $N \ge 0$ a non-negative integer. Then there exist symbols $a_j \in S(1,g)$, $j=0,1,...,N$, such that\\
\emph{(i)} $a_0(x,\xi)=\varphi(k(x,\xi))$ and $a_j(x,\xi)$ are supported in the support of $ \varphi(k(x,\xi))$ for any $j$. \\
\emph{(ii)} For every $1 \le q\le r \le \infty$ there exists $C_{qr}>0$ such that
$$
\norm{a_j(x,hD_x)}_{\L(L^q(\R^d),L^r(\R^d))} \le C_{qr} h^{-d(1/q-1/r)},  
$$
uniformly with respect to $h \in (0,1]$. \\
\emph{(iii)}There exists a constant $N_0 \ge 0$ such that, for all $1 \le q\le r \le \infty$, 
$$
\norm{\varphi(h^2H_0)-a(x,hD_x)}_{\L(L^q(\R^d),L^r(\R^d))} \le C_{Nqr}h^{N-N_0-d(1/q-1/r)}
$$
uniformly with respect to $h \in (0,1]$, where $a=\sum_{j=0}^N h^j a_j$. 
\end{lemma}

\begin{remark}
\label{2_remark_1}
We note that Assumption \ref{assumption_A} implies a stronger bounds on $a_j$:
$$|\dderiv{x}{\xi}{\alpha}{\beta}a_j(x,\xi)| \le C_{\alpha\beta}\<x\>^{-j-|\alpha|}\<\xi\>^{-|\beta|},$$
though we do not use this estimate in the following argument.
\end{remark}

We next recall the Littlewood-Paley decomposition in terms of $\varphi(h^2H_0)$. Consider a $4$-adic partition of unity with respect to $[1,\infty)$:
$$
\sum_{j=0}^\infty \varphi(2^{-2j}\lambda)=1,\quad \lambda \in [1,\infty),
$$
where $\varphi \in C^\infty_0(\R)$ with $\supp \varphi \subset [1/4,4]$ and $0 \le \varphi \le 1$. 
\begin{lemma}					
\label{2_lemma_2}
Let $\chi \in C^\infty_0(\R^d)$. Then, for all $2\le q < \infty$ with $0 \le d(1/2-1/q) \le 1$, 
\begin{align*}
\norm{(1-\chi)f}_{L^q(\R^d)}\lesssim \norm{f}_{L^2(\R^d)}
+ \left( \sum_{j = 0}^\infty\norm{(1-\chi)\varphi(2^{-2j}H_0)f}^2_{L^q(\R^d)} \right)^{\frac12}.
\end{align*}
\end{lemma}

This lemma can be proved similarly to the case of the Laplace-Beltrami operator on compact manifolds without boundaries (cf. \cite[Corollary 2.3]{BGT}). By using this lemma, we have the following:

\begin{proposition}
\label{2_proposition_1}
Let $\chi_0$ be as that in Theorem \ref{theorem_1}. 
Suppose that there exist $h_0,\delta>0$ small enough such that, for any $\psi \in C^\infty((0,\infty))$ and any admissible pair $(p,q)$ with $p>2$, 
\begin{align}
\label{2_proposition_1_1}
\norm{(1-\chi_0)\psi(h^2H_0)e^{-itH}u_0}_{L^p([-\delta,\delta];L^q(\R^d))} 
\le C \norm{u_0}_{L^2(\R^d)},
\end{align}
uniformly with respect to $h \in (0,h_0]$. Then, the statement of Theorem \ref{theorem_1} (i) holds. 
\end{proposition}

\begin{proof}
By Lemma \ref{2_lemma_2} with $f=e^{-itH}u_0$, the Minkowski inequality and the unitarity of $e^{-itH}$ on $L^2(\R^d)$, we have
\begin{align*}
&\norm{(1-\chi_0)e^{-itH}u_0}_{L^p([-\delta,\delta];L^q(\R^d))}\\
&\lesssim \norm{u_0}_{L^2(\R^d)}
+ \left( \sum_{j=0}^\infty\norm{(1-\chi_0)\varphi(2^{-2j}H_0)e^{-itH}u_0}^2_{L^p([-\delta,\delta];L^q(\R^d))} \right)^{1/2}.
\end{align*}
For $0 \le j \le [-\log h_0]+1$, we have the bound
\begin{align*}
&\sum_{j=0}^{[-\log h_0]+1}\norm{(1-\chi_0)\varphi(2^{-2j}H_0)e^{-itH}u_0}^2_{L^p([-\delta,\delta];L^q(\R^d))}\\
&\lesssim \sum_{j=0}^{[-\log h_0]+1}\norm{\varphi(2^{-2j}H_0)}_{\L(L^2(\R^d),L^q(\R^d))}\norm{e^{-itH}u_0}_{L^\infty([-\delta,\delta];L^2(\R^d))}\\
&\lesssim ([-\log h_0]+1)2^{([-\log h_0]+1)d(1/2-1/q)}\norm{u_0}_{L^2(\R^d)}.
\end{align*}
Choosing $\psi \in C_0^\infty(\R)$ with $\psi\equiv 1$ on $\supp \varphi$, we can write
\begin{align*}
&\varphi(h^2H_0)e^{-itH}\\
&=\psi(h^2H_0)e^{-itH}\varphi(h^2H_0)+\psi(h^2H_0)i\int_0^t e^{-i(t-s)H}[V,\varphi(h^2H_0)]e^{-isH}ds\\
&=\psi(h^2H_0)e^{-itH}\varphi(h^2H_0)+R(t,h).
\end{align*}
Since $[H,\varphi(h^2H_0)]=[V,\varphi(h^2H_0)]=O(h)$ on $L^2(\R^d)$, the remainder term $R(t,h)$ satisfies
\begin{equation}
\begin{aligned}
\label{2_proposition_1_proof_1}
&\sup_{0 \le t \le 1}\norm{R(t,h)}_{\L(L^2(\R^d),L^q(\R^d))}\\
&\lesssim \norm{\psi(h^2H_0)}_{\L(L^2(\R^d),L^q(\R^d))}\norm{[V,\varphi(h^2H_0)]}_{\L(L^2(\R^d))}\\
&\lesssim h^{-d(1/2-1/q)+1}.
\end{aligned}
\end{equation}
We here note that $\gamma:=-d(1/2-1/q)+1=-2/p+1>0$ since $p>2$. By \eqref{2_proposition_1_1}, \eqref{2_proposition_1_proof_1} with $h=2^{-j}$ and the almost orthogonality of $\supp \varphi(2^{-2j}\cdot)$, we obtain
\begin{align*}
&\sum_{j=[-\log h_0]}^\infty\norm{(1-\chi_0)\varphi(2^{-2j}H_0)e^{-itH}u_0}^2_{L^p([-\delta,\delta];L^q(\R^d))} \\
&\lesssim 
\sum_{j=[-\log h_0]}^\infty\left(\norm{\varphi(2^{-2j}H_0)u_0}^2_{L^2(\R^d)}+2^{-2\gamma j}\norm{u_0}^2_{L^2(\R^d)}\right)\\
&\lesssim \norm{u_0}^2_{L^2(\R^d)}, 
\end{align*}
Combining with the bound for $0 \le j \le [-\log h_0]+1$, we have
$$\norm{(1-\chi_0)e^{-itH}u_0}_{L^p([-\delta,\delta];L^q(\R^d))} 
\lesssim \norm{u_0}_{L^2(\R^d)}. 
$$
Finally, we split the time interval $[-T,T]$ into $([T/\delta]+1)$ intervals with size $2\delta$, and obtain
\begin{align*}
&\norm{(1-\chi_0)\psi(h^2H_0)e^{-itH}u_0}_{L^p([-T,T];L^q(\R^d))}\\
&\le \sum_{k=-[T/\delta]}^{[T/\delta]+1}\norm{(1-\chi_0)\psi(h^2H_0)e^{-itH}e^{-i(k+1)H}u_0}_{L^p([-\delta,\delta];L^q(\R^d))}\\
&\le C_T\norm{u_0}_{L^2(\R^d)}.
\end{align*}
  \end{proof}


\section{Isozaki-Kitada parametrix}
\label{Isozaki-Kitada_parametrix}
In this section we assume Assumption \ref{assumption_A} with $0<\mu=\nu<1/2$ 
without loss of generality, and construct the Isozaki-Kitada parametrix. Since the potential $V$ can grow at infinity, it is difficult to construct directly the Isozaki-Kitada parametrix for $e^{-itH}$ even though we restrict it in an outgoing or incoming region. To overcome this difficulty, we approximate $e^{-itH}$ as follows. Let $\rho \in C_0^\infty(\R^d)$ be a cut-off function such that $\rho(x)=1$ if $|x| \le 1$ and $\rho(x)=0$ if $|x| \ge 2$. For a small constant $\ep>0$ and $h \in (0,1]$, we define $H_h$ by
$$
H_h=H_0+V_h,\quad V_h=V(x)\rho(\ep h x). 
$$
We note that, for any fixed $\ep>0$, 
$$h^2|\partial_x^\alpha V_h(x)| \le C_{\alpha} h^2\<x\>^{2-\mu-|\alpha|} \le C_{\ep,\alpha}\<x\>^{-\mu-|\alpha|},\quad x \in \R^d,$$
where $C_{\ep,\alpha}$ may be taken uniformly with respect to $h \in (0,1]$. Such a type modification has been used to prove Strichartz estimates and local smoothing effects for Schr\"odinger equations with super-quadratic potentials (see, Yajima-Zhang \cite[Section 4]{Yajima_Zhang}).

For $R>0$, an open interval $J\Subset (0,\infty)$ and $-1<\sigma<1$, we define the outgoing and incoming regions by
\begin{align*}
\Gamma^\pm(R,J,\sigma):=\left\{(x,\xi) \in \R^{2d};|x|>R,\ |\xi| \in J,\ \pm \frac{x\cdot\xi}{|x||\xi|} > -\sigma\right\}, 
\end{align*}
respectively. Since $H_0+h^2V_h$ is a long-range perturbation of $-\Delta/2$, we have the following theorem due to Robert \cite{Robert} and Bouclet-Tzvetkov \cite{Bouclet_Tzvetkov_1}. 

\begin{theorem}						
\label{3_theorem_1}
Let $J,J_0,J_1$ and $J_2$ be relatively compact open intervals, $\sigma,\sigma_0,\sigma_1$ and $\sigma_2$ real numbers so that
$
J\Subset J_0\Subset J_1\Subset J_2 \Subset (0,\infty)$ and $
-1<\sigma<\sigma_0<\sigma_1<\sigma_2<1. 
$
Fix arbitrarily $\ep>0$. Then there exist $R_0>0$ large enough and $h_0>0$ small enough such that the followings hold.\\
\emph{(i)} There exist two families of smooth functions 
$$
\{S^+_h;h \in (0,h_0],R\ge R_0\},\ \{S^-_h;h \in (0,h_0],R\ge R_0\} \subset C^\infty(\R^{2d};\R)
$$ satisfying the Eikonal equation associated to $k+h^2V_h$:
$$
k(x,\partial_x S^\pm_h(x,\xi))+h^2V_h(x)=\frac12|\xi|^2,\quad 
(x,\xi) \in \Gamma^\pm(R^{1/4},J_2,\sigma_2),\quad h \in (0,h_0],
$$ 
respectively, such that
\begin{align}
\label{3_theorem_1_1}
|\dderiv{x}{\xi}{\alpha}{\beta}(S_h^\pm(x,\xi)-x\cdot\xi)| 
\le C_{\alpha\beta}\<x\>^{1-\mu-|\alpha|},\quad
\alpha,\beta \in \Z_+^d,\ 
x,\xi \in \R^d,
\end{align}
where $C_{\alpha\beta}>0$ may be taken uniformly with respect to $R$ and $h$.\\
\emph{(ii)} For every $R \ge R_0$, $h \in (0,h_0]$ and $N=0,1,...$, we can find 
$$
b^\pm_h=\sum_{j=0}^N h^j b_{h,j}^\pm \in S(1,g)\quad \text{with}\quad \supp b_{h,j}^\pm \subset \Gamma^\pm(R^{1/3},J_1,\sigma_1)
$$
such that, for every $a^\pm \in S(1,g)$ with $\supp a^\pm \subset\Gamma^\pm(R,J,\sigma)$, there exist 
$$
c^\pm_h=\sum_{j=0}^N h^j c_{h,j}^\pm(h) \in S(1,g)\quad \text{with}\quad \supp c_{h,j}^\pm \subset \Gamma^\pm(R^{1/2},J_0,\sigma_0)
$$
such that, for all $\pm t \ge0$, 
$$e^{-ithH_h}a^\pm(x,hD)=\F_{{\IK}}(S_h^\pm,b^\pm_h)e^{ith\Delta/2}\F_{{\IK}}(S^\pm_h,c^\pm_h)^*+Q_{{\IK}}^\pm(t,h,N),$$
 respectively, where $\F_{{\IK}}(S_h^\pm,w)$ are Fourier integral operators defined by 
$$
\F_{{\IK}}(S_h^\pm,w)f(x)=\frac{1}{(2\pi h)^{d}}\int e^{i(S^\pm_h(x,\xi)-y\cdot\xi)/h}w(x,\xi)f(y)dyd\xi,
$$
respectively. Moreover, for all $s \in \R$ there exists $C_N>0$ such that
\begin{align}
\label{3_theorem_1_2}
\norm{(h^2H_h+L)^{s}Q_{{\IK}}^\pm(t,h,N)}_{\L(L^2(\R^d))} \le C_Nh^{N-1}
\end{align}
uniformly with respect to $h \in (0,h_0]$ and $0 \le \pm t \le h^{-1}$, where $L>1$, independent of $h$, $t$ and $x$, is a large constant so that $h^2V_h+L \ge 1$. \\
\emph{(iii)} The distribution kernels $K_{{\IK}}^\pm(t,h,x,y)$ of $\F_{{\IK}}(S_h^\pm,b^\pm_h)e^{-ith\Delta/2}\F_{{\IK}}(S_h^\pm,c^\pm_h)^*$ satisfy dispersive estimates:
\begin{align}
\label{dispersive_1}
|K_{{\IK}}^\pm(t,h,x,y)| \le C|th|^{-d/2}, 
\end{align}
for any $h \in (0,h_0]$, $0\le \pm t \le h^{-1}$ and $x,\xi \in \R^d$, respectively.
\end{theorem}

\begin{proof}
This theorem is basically known, and we only check \eqref{3_theorem_1_2} for the outgoing case. For the detail of the proof, we refer to \cite[Section 4]{Robert} and \cite[Section 3]{Bouclet_Tzvetkov_1}. We also refer to the original paper by Isozaki-Kitada \cite{Isozaki_Kitada}. 

The remainder $Q_{{\IK}}^+(t,h,N)$ consists of the following three parts:
\begin{align*}
&-h^{N+1}e^{-ithH_h}q_1(h,x,hD),\\
&-ih^N\int_0^t e^{-i(t-\tau)hH_h}\F_{{\IK}}^+(S_h^+,q_2(h))e^{i\tau h\Delta/2}\F_{{\IK}}^+(S_h^+,c^+_h)^*d\tau,\\
&-(i/h)\int_0^t e^{-i(t-\tau)hH_h}\wtilde{Q}(\tau,h)d\tau, 
\end{align*}
where $\{q_1(h,\cdot,\cdot),q_2(h,\cdot,\cdot);h \in (0,h_0]\}\subset \bigcap_{M = 1}^\infty S(\<x\>^{-N}\<\xi\>^{-M},g)$ is a bounded set, and $\wtilde{Q}(s,h)$ is a integral operator with a kernel $\tilde{q}(s,h,x,y)$ satisfying
$$
|\dderiv{x}{\xi}{\alpha}{\beta}\tilde{q}(\tau,h,x,y)| \le C_{\alpha\beta} h^{M-|\alpha+\beta|}(1+|\tau|+|x|+|y|)^{-M+|\alpha+\beta|},\ \tau\ge 0,
$$
for any $M\ge 0$. A standard $L^2$-boundedness of $h$-PDO and FIO then imply 
$$
\norm{(h^2H_0+1)^{s}\left(q_1(h,x,hD)+\F_{{\IK}}^+(S_h^+,q_2(h))\right)}_{\L(L^2(\R^d))}\le C_s,
$$
and a direct computation yields
$$
\norm{(h^2H_0+1)^{s}\wtilde{Q}(\tau,h)}_{\L(L^2(\R^d))}\le C_Mh^M. 
$$
On the other hand, if we choose a constant $L>0$ so large that $h^2V_h+L\ge 1$, then we have
\begin{align}
\label{3_theorem_1_proof_3}
\norm{(h^2H_h+L)^s(h^2H_0+1)^{-s}}_{\L(L^2(\R^d))} \le C_s. 
\end{align}
Indeed, if $s$ is a positive integer, then \eqref{3_theorem_1_proof_3} is obvious since $h^2V_h+L \lesssim 1$. For any negative integer $s$, \eqref{3_theorem_1_proof_3} follows from the fact that $h^2H_0+1 \le h^2H_h+L$. For general $s\in\R$, we obtain \eqref{3_theorem_1_proof_3} by an interpolation. \eqref{3_theorem_1_2} follows from the above three estimates since $(h^2H_h+L)^s$ commutes with $e^{-ithH_h}$.
  \end{proof}

The following key lemma tells us that one can still construct the Isozaki-Kitada parametrix of the original propagator $e^{-ithH}$ if we restrict the support of initial data in the region $\{x;|x| <h^{-1}\}$.
\begin{lemma}							
\label{3_lemma_1}
Suppose that $\{a_h^\pm\}_{h \in (0,1]}$ are bounded sets in $S(1,g)$ and satisfy
$$\supp a_h^\pm \subset \Gamma^\pm(R,J,\sigma)\cap \{x;|x|<h^{-1}\}, $$
respectively. Then for any $M \ge 0$, $h \in (0,h_0]$ and $0 \le \pm t \le h^{-1}$, we have 
$$\norm{(e^{-ithH}-e^{-ithH_h})a_h^\pm(x,hD)}_{\L(L^2(\R^d))} \le C_M h^M,$$
where $C_M>0$ is independent of $h$ and $t$.
\end{lemma}

\begin{proof}
We prove the lemma for the outgoing case only, and the proof of incoming case is completely analogous. We set $A=a_h^+(x,hD)$ and $W_h=V-V_h$. The Duhamel formula yields
\begin{align*}
&(e^{-ithH}-e^{-ithH_h})A\\
&=-ih \int_0^t e^{-i(t-s)hH}W_he^{-ishH_h}Ads\\
&=-i h\int_0^t e^{-i(t-s)hH}e^{-ishH_h}W_hAds\\
&\ \ \ -h^2\int_0^t e^{-i(t-s)hH}\int_0^s e^{-i(s-\tau)hH_h}[H_0,W_h]e^{-i\tau hH_h}Ad\tau ds.
\end{align*}
Since $\supp a_h^+(\cdot,\xi) \subset \{x; |x| < h^{-1}\}$, we learn 
$
\supp W_h \cap a^+_h(\cdot,\xi)=\emptyset
$ if $\ep<1$. Combining with the asymptotic formula \eqref{symbolic_calculus}, this support property implies 
$$
\norm{W_hA}_{\L(L^2(\R^d))} \le C_Mh^M
$$
for any $M \ge0$. A direct computation yields that $[H_0,W_h]$ is of the form 
$$
\sum_{|\alpha|=0,1} a_\alpha(x)\partial_x^\alpha,\quad
 \supp a_\alpha \subset \supp W_h,\quad
|\partial_x^\beta a_\alpha(x)|\le C_{\alpha\beta}\<x\>^{-\mu+|\alpha|-|\beta|}.
$$ 
The support property of $W_h$ again yields 
$$
\norm{[H_h,[H_0,W_h]]A}_{\L(L^2(\R^d))} \le C_Mh^M.
$$
We next consider $[H_h,[K,W_h]]$ which has the form
$$
\sum_{|\alpha|=1,2} b_\alpha(x)\partial_x^\alpha+W_1(x),\ 
$$
where $b_\alpha$ and $W_1$ are supported in $\supp W_h$ and satisfy
$$
|\partial_x^\beta b_\alpha(x)|\le C_{\alpha\beta}\<x\>^{-2-\mu+|\alpha|-|\beta|},\quad
|\partial_x^\beta W_1(x)|\le C_{\alpha\beta}\<x\>^{2-2\mu}. 
$$
 Setting $I_1=\sum_{|\alpha|=1,2} b_\alpha(x)\partial_x^\alpha$ and $N_\mu:=[1/\mu]+1$, we iterate this procedure $N_\mu$ times with $W_h$ replaced by $W_1$. $(e^{-ithH}-e^{-ithH_h})A$ then can be brought to a linear combination of the following forms  (modulo $O(h^M)$ on $L^2(\R^d)$):
\begin{align*}
\int_{t \ge s_1 \ge \cdots\ge s_j \ge 0} e^{-i(t-s_1)hH}e^{-i(s_1-s_j)hH_h}I_{j/2}e^{-is_jhH_h}Ads_j\cdots ds_1
\end{align*}
for $j=2m,\ m=1,2,..,N_\mu$, and
\begin{align*}
\int_{t \ge s_1 \ge \cdots \ge s_{N_\mu} \ge 0}e^{-i(t-s_1)hH}e^{-i(s_1-s_{N_\mu})hH_h}W_{N_\mu}e^{-is_{N_\mu}hH_h}Ads_{2N_\mu}\cdots ds_1,
\end{align*}
where $I_k$ are second order differential operators with smooth and bounded coefficients, and $W_{N_\mu}$ is a bounded function since $2-2\mu N_\mu<0$. Moreover, they are supported in $\{x;|x|>(\ep h)^{-1}\}$. 
Therefore, it is sufficient to show that, for any $h \in (0,h_0]$, $0 \le \tau \le h^{-1}$, $\alpha \in \Z^d_+$ and $M \ge 0$, 
\begin{align}
\label{3_lemma_1_1}
\norm{(1-\rho(\ep hx))\partial^\alpha_x e^{-i\tau hH_h}A}_{\L(L^2(\R^d))} \le C_{M,\alpha} h^{M-|\alpha|}. 
\end{align}
We now apply Theorem \ref{3_theorem_1} to $e^{-i\tau hH_h}A$ and obtain
\begin{align*}
e^{-i\tau hH_h}A=\F_{{\IK}}(S_h^+,b^+_h)e^{i\tau h\Delta/2}\F_{{\IK}}(S_h^+,c^+_h)^*+Q_{{\IK}}^+(t,h,N).
\end{align*}
Recall that the elliptic nature of $H_0$ implies, for every $s \ge0$, 
\begin{align*}
\norm{\<D\>^s(h^2H_0+1)^{-s/2}f}_{L^2(\R^d)} &\le Ch^{-s}\norm{f}_{L^2(\R^d)},\\
\norm{(h^2H_0+1)^{s/2}(h^2H_h+L)^{-s/2}f}_{L^2(\R^d)}&\le C\norm{f}_{L^2(\R^d)},
\end{align*}
if $L>0$ so large that $h^2H_h+L \ge 1$. Combining these estimates with \eqref{3_theorem_1_2}, the remainder satisfies
$$
\norm{\<D\>^s Q_{{\IK}}^+(t,h,N)f}_{L^2(\R^d)} \le C_{N,s}h^{N-1-s}\norm{f}_{L^2(\R^d)},\quad s \ge 0.
$$

The main term can be handled in terms of the non-stationary phase method as follows. The distribution kernel of the main term is given by 
\begin{align}
\label{3_lemma_1_2}
(2\pi h)^{-d}(1-\rho(\ep h x))\partial_x^\alpha \int e^{i\Phi^+_h(\tau,x,y,\xi)/h}b_h^+(x,\xi)\overline{c^+_h(y,\xi)}d\xi, 
\end{align}
where $\Phi^+_h(\tau,x,y,\xi)=S^+_h(x,\xi)-\frac12\tau|\xi|^2-S^+_h(y,\xi)$. We here claim that
\begin{align}
\label{3_lemma_1_3}
\supp c^+_h\subset \{(x,\xi)\in \R^{2d};a^+_h(x,\partial_\xi S^+_h(x,\xi))\neq 0\}.
\end{align}
This property follows from the construction of $c_j^+(h)$, $j=0,1,...,N$. We set 
$$
\wtilde{S}^+_h(x,y,\xi)=\int_0^1\partial_xS^+_h(y+\theta(x-y),\xi)d\theta. 
$$ 
Let $\xi \mapsto [\wtilde{S}^+_h]^{-1}(x,y,\xi)$ be the inverse map of $\xi \mapsto \wtilde{S}^+_h(x,y,\xi)$, and we denote their Jacobians by $A_1=|\det \partial_\xi\wtilde{S}^+_h(x,y,\xi)|$ and $A_2=|\det \partial_\xi[\wtilde{S}^+_h]^{-1}(x,y,\xi)|$, respectively. $c_j^+(h)$ then satisfy the following triangular system:
\begin{align*}
\overline{c_{h,j}^+(x,\xi)}=b_{h,0}^+(x,\xi)^{-1}\left(r_{h,j}^+(x,\wtilde{S}^+_h(x,y,\xi))A_1\right)\bigg|_{y=x},\quad j=0,1,...,N,
\end{align*}
where $r_{h,0}^+=a_h^+(x,\wtilde{S}^+_h(x,y,\xi))$ and $r_j^+$, $j\ge 1$, is a linear combination of $$
\frac{1}{i^{|\alpha|}\alpha !}
\left(
\partial_\xi^\alpha \partial_y^\alpha b_{h,k_0}^+(x,[\wtilde{S}^+_h]^{-1}(x,y,\xi))c_{h,k_1}^+(y,[\wtilde{S}^+_h]^{-1}(x,y,\xi))A_2
\right)\bigg|_{y=x},
$$
where $\alpha \in \Z_+^d$ and $k_0,k_1=0,1,...,j$ so that $0\le |\alpha| \le j,\ k_0+k_1=j-|\alpha|$ and $k_1 \le j-1$. 
Therefore, we inductively obtain 
$$
\supp c_{h,0}^+ \subset \supp r_0^+|_{y=x},\quad\supp c_{h,j}^+ \subset \supp c_{h,j-1}^+(h),\ j=1,2,...,N,
$$ and \eqref{3_lemma_1_3} follows. 
In particular, $c^+_h$ vanishes in the region $\{x;|x| \ge h^{-1}\}$. 
By using \eqref{3_theorem_1_1}, we have 
$$\partial_\xi\Phi^+_h(\tau,x,y,\xi)=(x-y)(\Id+O(R^{-\mu/3}))-\tau\xi,$$
which implies
$$|\partial_\xi\Phi^+_h(\tau,x,y,\xi)| \ge \frac{|x|}{2}-|y|-|\tau\xi|$$
as long as $R \ge 1$ large enough. 
We now set 
$$\ep=\frac{1}{2(\sup J_2)^{1/2}+2}.$$
Since $|x|>(\ep h)^{-1}$, $ |y|<h^{-1}$ and $|\xi|^2 \in J_2$ 
on the support of the amplitude, we have
$$|\partial_\xi\Phi^+_h(\tau,x,y,\xi)| >c (|x|+h^{-1})>c(1+|x|+|y|+|\tau|),\quad 0 \le \tau \le h^{-1},$$
for some $c>0$ independent of $h$. Therefore, integrating by parts \eqref{3_lemma_1_2} with respect to 
$$-ih\abs{\partial_\xi\Phi^+_h}^{-2}(\partial_\xi\Phi^+_h)\cdot\partial_\xi, $$
we obtain
\begin{align*}
&\left|(2\pi h)^{-d}(1-\rho(\ep h x))\partial_x^\alpha \partial_y^\beta \int e^{i\Phi^+_h(\tau,x,y,\xi)/h}b_h^+(x,\xi)\overline{c_h^+(y,\xi)}d\xi\right|\\
&\le C_{\alpha \beta M}h^{M-d-|\alpha+\beta|}(1+|x|+|y|+\tau)^{-M},
\end{align*}
for all $M \ge0$, $0 \le \tau\le h^{-1}$ and $\alpha,\beta \in \Z^d_+$. \eqref{3_lemma_1_1} then follows from the $L^2$-boundedness of FIOs. 
  \end{proof}


\section{WKB parametrix}
\label{WKB_parametrix}
In the previous section we proved that $e^{-ithH}$ is well approximated in terms of an Isozaki-Kitada parametrix on a time scale of order $h^{-1}$ if we localize the initial data in regions $\Gamma^\pm(R,J,\sigma)\cap\{x;R<|x| < h^{-1}\}$. Therefore, it remains to control $e^{-ithH}$ on a region $\{x;|x|\gtrsim h^{-1}\}$. In this section we construct the WKB parametrix for $e^{-ithH}a(x,hD)$, where $a \in S(1,g)$ with $\supp a \subset  \{(x,\xi) \in \R^{2d};|x|\gtrsim h^{-1},\ |\xi|^2\in J\}$. In what follows we assume that $H$ satisfies Assumption \ref{assumption_A} with $\mu=0$ and $\nu=1$. 

We first consider the phase function of the WKB parametrix, that is a solution to the time-dependent Hamilton-Jacobi equation generated by $p_h(x,\xi)=k(x,\xi)+h^2V(x)$. For $R>0$ and open interval $J\Subset (0,\infty)$, we set $$
\Omega(R,J):=\{(x,\xi)\in \R^{2d};|x| > R/2,\ |\xi|^2 \in J\}.
$$ 
We note that $\Omega(R_1,J_1) \subset \Omega(R_2,J_2)$ if $R_1>R_2$ and  $J_1\subset J_2$.
\begin{proposition}						
\label{WKB_2_proposition_1}
Choose arbitrarily an open interval $J\Subset (0,\infty)$. Then, there exist $\delta_0>0$ and $h_0>0$ small enough such that, for all $h\in (0,h_0]$, $0<R\le h^{-1}$ and $0<\delta\le\delta_0$, we can construct a family of smooth functions
$$\{\Psi_h(t,x,\xi)\}_{h \in (0,h_0]} \subset C^\infty((-\delta R,\delta R)\times\R^{2d})$$
such that $\Psi_h(t,x,\xi)$ satisfies the Hamilton-Jacobi equation associated to $p_h$:
\begin{equation}\left\{\begin{aligned}
\label{WKB_2_proposition_1_1}
\partial_t \Psi_h(t,x,\xi)
&=-p_h(x,\partial_x\Psi_h(t,x,\xi)),\quad 
0<|t|<\delta R,\ (x,\xi) \in \Omega(R,J), \\
\Psi_h(0,x,\xi)
&=x\cdot\xi,\quad 
(x,\xi)\in \Omega(R,J). 
\end{aligned}\right.\end{equation}
Moreover, for all $|t|\le \delta R$ and $\alpha,\beta\in \Z^d_+$, $\Psi_h(t,x,\xi)$ satisfies 
\begin{align}
\label{WKB_2_proposition_1_2}
&|\dderiv{x}{\xi}{\alpha}{\beta}(\Psi_h(t,x,\xi)-x\cdot\xi)|
\le C\delta R^{1-|\alpha|},\quad x,\xi \in \R^d,\ |\alpha+\beta|\ge2,\\
\label{WKB_2_proposition_1_4}
&|\dderiv{x}{\xi}{\alpha}{\beta}\left(\Psi_h(t,x,\xi)-x\cdot\xi+tp_h(x,\xi)\right)|
\le C_{\alpha\beta} \delta R^{|\alpha|}|t|, \quad x,\xi \in \R^d.
\end{align} 
\end{proposition}

\begin{proof}
We give the proof in Appendix \ref{appendix_A}. 
  \end{proof}
We next define the corresponding FIO. Let $0<R \le h^{-1}$, $J\Subset J_1\Subset (0,\infty)$ open intervals and $\Psi_h$ defined by the previous proposition with $R,J$ replaced by $R/4,J_1$, respectively. Suppose that $\{a_h(t,\cdot,\cdot)\}_{h \in (0,h_0],0 \le t \le \delta R}$ is bounded in $S(1,g)$ and supported in $\Omega(R,J)$. We then define the FIO for WKB parametrix $\F_{{\WKB}}(\Psi_h(t),a_h(t)):\S(\R^d)\to\S(\R^d)$ by
$$
\F_{{\WKB}}(\Psi_h(t),a_h(t))u(x)=\frac{1}{(2\pi h)^d}\int e^{i(\Psi_h(t,x,\xi)-y\cdot\xi)/h}a_h(t,x,\xi)u(y)dyd\xi.
$$

\begin{lemma}							
\label{WKB_lemma_1}
$\F_{{\WKB}}(\Psi_h(t),a(t))$ is bounded on $L^2(\R^d)$ uniformly with respect to $R$, $h$ and $t$:
$$
\sup_{h \in (0,h_0],0 \le t \le \delta R}\norm{\F_{{\WKB}}(\Psi_h(t),a(t))}_{\L(L^2(\R^d))} \le C.
$$
\end{lemma}

\begin{proof}
For $|t| \le \delta R$, we define the map $\wtilde{\Xi}(t,x,\xi,y)$ on $\R^{3d}$ by
$$\wtilde{\Xi}(t,x,y,\xi)=\int_0^1 (\partial_x \Psi_h)(t,y+\lambda(x-y),\xi)d\lambda.$$
By \eqref{WKB_2_proposition_1_2}, $\wtilde{\Xi}(t,x,y,\xi)$ satifies
$$
\abs{\partial_x^\alpha\dderiv{y}{\xi}{\beta}{\gamma}(\wtilde{\Xi}(t,x,y,\xi)-\xi)}
\le C_{\alpha\beta\gamma}\delta R^{-|\alpha+\beta|},\quad |t| \le \delta R,\ x,y \in \R^d,
$$
and the map $\xi \mapsto \wtilde{\Xi}(t,x,\xi,y)$ hence is a diffeomorphism from $\R^d$ onto itself for all $|t| \le \delta R$ and $x,y \in \R^d$, provided that $\delta>0$ is small enough. 
Let $ \xi \mapsto [\wtilde{\Xi}]^{-1}(t,x,y,\xi)$ be the corresponding inverse. $[\wtilde{\Xi}]^{-1}$ satisfies the same estimate as that for $\wtilde{\Xi}$:
$$
\abs{\dderiv{x}{y}{\alpha}{\beta}\partial_\xi^\gamma([\wtilde{\Xi}]^{-1}(t,x,y,\xi)-\xi)}
\le C_{\alpha\beta\gamma}\delta R^{-|\alpha+\beta|}
\quad\text{on}\quad [-\delta R,\delta R]\times \R^{3d}. 
$$
Using the change of variables $\xi \mapsto [\wtilde{\Xi}]^{-1}$, $\F_{{\WKB}}(\Psi_h(t),a(t))\F_{{\WKB}}(\Psi_h(t),a(t))^*$ can be regarded as a semi-classical PDO with a smooth and bounded amplitude
$$
a_h(t,x,[\wtilde{\Xi}]^{-1}(t,x,y,\xi))\overline{a_h(t,y,[\wtilde{\Xi}]^{-1}(t,x,y,\xi))}|\det \partial_\xi[\wtilde{\Xi}]^{-1}(t,x,y,\xi)|. 
$$
Therefore, the $L^2$-boundedness follows from the Calder\'on-Vaillancourt theorem. 
  \end{proof}

We now state the main result in this section.

\begin{theorem}						
\label{A_theorem_1}
Let $J \Subset J_0 \Subset J_1 \Subset (0,\infty)$ be open intervals. Then there exist $\delta_0,h_0>0$ such that, for all $h\in (0,h_0]$, $0 <R \le h^{-1}$, $0<\delta\le \delta_0$ and all symbol
$$
a \in S(1,g)
\quad\text{with}\quad
\supp a \in \Omega(R,J),
$$ 
and all $N \ge 0$, we can find a semi-classical symbol 
$$
b_h(t,x,\xi)=\sum_{j=0}^Nh^jb_{h,j}(t,x,\xi)
$$ 
with $b_{h,j}(t,\cdot,\cdot)$ bounded in $S(1,g)$ and $\supp b_{h,j}(t,\cdot,\cdot) \subset \Omega(R/2,J_0)$ uniformly with respect to $h \in (0,h_0]$ and $|t|\le \delta R$ such that $e^{-ithH}a(x,hD_x)$ can be brought to the form
\begin{align*}
e^{-ithH}a(x,hD_x)=\F_{{\WKB}}(\Psi_h(t),b_h(t))+Q_{{\WKB}}(t,h,N),
\end{align*}
where $\F_{{\WKB}}(\Psi_h(t),b_h(t))$ is the Fourier intehgral operator with the phase function $\Psi_h(t,x,\xi)$, defined in Proposition \ref{WKB_2_proposition_1} with $R,J$ replaced by $R/4,J_1$ respectively, and its distribution kernel satisfies the dispersive estimates:
\begin{align}
\label{dispersive_2}
|K_{{\WKB}}(t,h,x,y)| \le C |th|^{-d/2},\quad h \in (0,h_0],\ 0<|t| \le \delta R,\ x,\xi \in \R^d.
\end{align}
Moreover the remainder $Q_{{\WKB}}(t,h,N)$ satisfies
\begin{align*}
\norm{Q_{{\WKB}}(t,h,N)}_{\L(L^2(\R^d))}
\le C_N h^{N}|t|,\quad h \in (0,h_0],\ |t| \le \delta R.
\end{align*}
Here the constants $C,C_N>0$ can be taken uniformly with respect to $h$, $t$ and $R$. 
\end{theorem}

\begin{remark}
The essential point of Theorem \ref{A_theorem_1} is to construct the parametrix on the time interval $|t| \le \delta R$. When $|t|>0$ is small and independent of $R$, such a parametrix construction is basically well known (cf. \cite{Robert2}). 
\end{remark}

\begin{bew}{of Theorem \ref{A_theorem_1}}
We consider the case when $t \ge0$ and the proof for $t<0$ is similar. \\
\textbf{Construction of the amplitude.} The Duhamel formula yields
\begin{align*}
&e^{-ithH}\F_{{\WKB}}(\Psi_h(0),b_h(0))\\
&=\F_{{\WKB}}(\Psi_h(t),b_h(t))+\frac ih \int_0^t  e^{-i(t-s)hH}(hD_s+h^2H)\F_{{\WKB}}(\Psi_h(s),b_h(s))ds.
\end{align*}
Therefore, it suffices to show that there exist $b_{h,j}$ with $b_{h,0}|_{t=0}=a$ and $b_{h,j}|_{t=0}=0$ for $j \ge 1$ such that
\begin{align}
\label{WKB_theorem_proof_0}
\norm{(hD_s+h^2H)\F_{{\WKB}}(\Psi_h(s),b_h(s))}_{\L(L^2)} 
\le C_N h^{N+1},\quad 
0 \le s \le \delta R. 
\end{align}
Let $k+k_1$ be the full symbol of $H_0$: $H_0=k(x,D)+k_1(x,D)$, and define a smooth vector field $\X_h(t)$ and a function $\Y_h(t)$ by
$$
\X_h(t,x,\xi):=(\partial_\xi k)(x,\partial_x \Psi_h(t,x,\xi)),\quad \Y_h(t,x,\xi):=[(k+k_1)(x,\partial_x)\Psi_h](t,x,\xi).
$$
Symbols $\{b_{h,j}\}$ can be constructed in terms of the method of characteristics as follows. For all $0 \le s,t \le \delta R$, we consider the flow $z_h(t,s,x,\xi)$ generated by $\X_h(t)$, that is the solution to the following ODE:
\begin{align*}
\partial_t z_h(t,s,x,\xi)=\X_h(z_h(t,s,x,\xi),\xi);\quad z_h(s,s)=x.
\end{align*}
Choose $R',R''$ and two intervals $J_0',J_0''$ so that 
$$R/2>R'>R''>R/4,\quad J_0\Subset J_0'\Subset J_0''\Subset (0,\infty).$$
\eqref{WKB_2_proposition_1_4} and the same argument as that in the proof of Lemmas \ref{WKB_1_lemma_1} and \ref{WKB_1_lemma_3} imply that there exists $\delta_0,h_0>0$ small enough such that, for all $0<\delta\le \delta_0$, $h \in (0,h_0]$ $0<R\le h^{-1}$ and $0 \le s,t \le \delta R$, $z_h(t,s)$ is well defined on $\Omega(R'',J_0'')$ and satisfies
\begin{equation}
\begin{aligned}
\label{WKB_theorem_proof_1}
|\dderiv{x}{\xi}{\alpha}{\beta}(z_h(t,s,x,\xi)-x)| 
\le C_{\alpha\beta}\delta R^{1-|\alpha|}. 
\end{aligned}
\end{equation}
In particular, $(z_h(t,s,x,\xi),\xi) \in \Omega(R',J')$ for $0 \le s,t \le \delta R$ if $\delta>0$, depending only on $J''$, is small enough. We now define $\{b_{h,j}(t,x,\xi)\}_{0 \le j \le N}$ inductively by
\begin{equation}
\begin{aligned}
\nonumber
&b_{h,0}(t,x,\xi)=a(z_h(0,t),\xi)\exp\left(\int_0^t\Y_h(s,z_h(s,t,x,\xi),\xi)ds\right),\\
&b_{h,j}(t,x,\xi)\\
&=-\int_0^t (iH_0 b_{h,j-1})(s,z_h(s,t),\xi) \exp\left(\int_u^t\Y_h(u,z_h(u,t,x,\xi),\xi)du\right)ds. 
\end{aligned}
\end{equation}
Since $\supp a \in \Omega(R,J)$ and $z_h(t,s,\Omega(R,J)) \subset \{x;|x|>R/2\}$ for all $0 \le s,t \le \delta R$, $b_{h,j}(t)$ are supported in $\Omega(R/2,J_0)$. Thus, if we extend $b_{h,j}$ on $\R^{2d}$ so that $$b_{h,j}(t,x,\xi)=0,\quad(x,\xi)\notin \Omega(R/2,J_0),$$ then $b_{h,j}$ is still smooth in $(x,\xi)$. 
By \eqref{WKB_2_proposition_1_4} and \eqref{WKB_theorem_proof_1}, we learn
$$|\dderiv{x}{\xi}{\alpha}{\beta}\Y_h(s,z_h(s,t,x,\xi),\xi)| \le C\delta R^{-1-|\alpha|},\quad 0\le s,t\le\delta R.$$
$\{b_{h,j}(t,\cdot,\cdot);h \in (0,h_0],\ 0<R\le h^{-1},\ t \in [0,\delta R],\ 0 \le j \le N\}$ thus is a bounded set in $S(1,g)$ and $\supp b_{h,j}(t,\cdot,\cdot) \subset \Omega(R/2,J_0)$ uniformly with respect to $h \in (0,h_0]$ and $0 \le t \le \delta R$. 

A standard Hamilton-Jacobi theory shows that $b_{h,j}(t)$ satisfy the following transport equations:
\begin{equation}
\left\{
\begin{aligned}
&\partial_t b_{h,0}(t)+\X_h(t)b_{h,0}(t)+\Y_h(t)b_{h,0}(t)=0,\\
&\partial_t b_{h,j}(t)+\X_h(t)b_{h,j}(t)+\Y_h(t)b_{h,j}(t)=-iH_0b_{h,j-1}(t),\ j \ge 1,
\end{aligned}
\right.
\end{equation}
with the initial condition $b_{h,0}(0)=a$, $b_{h,j}(0)=0$, $j=1,2,...,N$. 
A direct computation then yields 
$$
e^{-i\Psi_h(s,x,\xi)/h}(hD_s+h^2H)\left(e^{i\Psi_h(s,x,\xi)/h}\sum_{j=0}^N h^j b_{h,j}\right)
=O(h^{N+1})\ \text{in}\ S(1,g)
$$
which, combined with Lemma \ref{WKB_lemma_1}, implies \eqref{WKB_theorem_proof_0}.\\
\textbf{Dispersive estimates.} The distribution kernel of $\F_{{\WKB}}(\Psi_h(t),b_h(t))$ is given by
$$K_{{\WKB}}(t,h,x,y)=\frac{1}{(2\pi h)^d}\int e^{\frac{i}{h}(\Psi_h(t,x,\xi)-y\cdot\xi)}b_h(t,x,\xi)d\xi. $$
Since $b_h(t,x,\xi)$ has a compact support with respect to $\xi$, 
$$|K_{{\WKB}}(t,h,x,y)| \le Ch^{-d}\le C|th|^{-d/2}\quad \text{for}\quad 0 < t \le h.$$
We hence assume $h<t$ without loss of generality. Choose $\chi \in S(1,g)$ so that $0 \le \chi \le 1$, $\chi \equiv 1$ on $\Omega(R/2,J_0)$ and $\supp \chi \subset \Omega(R/4,J_1)$, and set 
$$
\psi_h(t,x,y,\xi)
=\frac{(x-y)}{t}\cdot\xi-p_h(x,\xi)+\chi(x,\xi)\left(\frac{\Psi_h(t,x,\xi)-x\cdot\xi}{t}+p_h(x,\xi)\right).
$$
By the definition, we obtain
\begin{align*}
\psi_h(t,x,y,\xi)&=\frac{\Psi_h(t,x,\xi)-y\cdot\xi}{t},\quad t \in [h,\delta R],\ (x,\xi) \in \Omega(R/2,J_1),\ y \in \R^d, 
\end{align*}
and \eqref{WKB_2_proposition_1_4} implies 
$$\abs{\dderiv{x}{\xi}{\alpha}{\beta}\psi_h(t,x,y,\xi)}\le C_{\alpha\beta}\quad\text{on}\quad [0,\delta R]\times \R^{3d},\quad |\alpha+\beta|\ge 2.$$
Moreover, $\partial_\xi^2\psi_h(t,x,y,\xi)$ can be brought to the form
$$\partial_\xi^2\psi_h(t,x,y,\xi)=-(a^{jk}(x))_{j,k}+Q_h(t,x,\xi),$$
where the error term $Q_h(t,x,\xi)$ is a $d\times d$-matrix satisfying
$$
\abs{\dderiv{x}{\xi}{\alpha}{\beta}Q_h(t,x,\xi)} \le C_{\alpha\beta}\delta h^{|\alpha|}\quad\text{on}\quad [0,\delta R]\times \R^{2d}.
$$
Since $(a^{jk}(x))$ is uniformly elliptic,  the stationary phase theorem implies that
$$|K_{{\WKB}}(t,h,x,y)| \le Ch^{-d}|t/h|^{-d/2}=C|th|^{-d/2},$$
provided that $\delta>0$ is small enough. We complete the proof.
  \end{bew}

\section{Proof of Theorem \ref{theorem_1} (i)}
\label{Proof}
In this section we complete the proof of Theorem \ref{theorem_1} (i). Let $\chi_0 \in C^\infty_0(\R^d)$ with $\chi_0\equiv 1$ on $\{|x|<R_0\}$ and $\psi \in C_0^\infty((0,\infty))$. A partition unity argument and Lemma \ref{2_lemma_1} show that there exist $a^\pm \in S(1,g)$ with $\supp a^\pm \subset \Gamma^\pm(R_0,J,1/2)$ such that $(1-\chi_0)\psi(h^2H_0)$ is approximated by $a^\pm(x,hD)$:
$$
(1-\chi_0)\psi(h^2H_0)=a^+(x,hD)^*+a^-(x,hD)^*+Q_0(h), 
$$
where $J\Subset (0,\infty)$ is an open interval satisfying $\pi_\xi(\supp \varphi\circ k) \Subset J$ and $Q_0(h)$ satisfies 
$$
\sup_{h\in (0,1]}\norm{Q_0(h)}_{\L(L^2(\R^d),L^q(\R^d))} \le C_q,\quad q\ge 2. 
$$
Let $b\in C^\infty_0(\R^d;\R)$ be a cut-off function such that $b\equiv 1$ on a neighborhood of $J$. By the asymptotic formula \eqref{symbolic_calculus}, we can write
$$
a^\pm(x,hD)^*=b(hD)a^\pm(x,hD)^*+Q_1(h)
$$
where $Q_1(h)$ satisfies the same $\L(L^2,L^q)$-estimate as that of $Q_0(h)$. Therefore, 
\begin{align}
\label{proof_0}
\norm{(Q_0(h)+Q_1(h))e^{-itH}u_0}_{L^p([-\delta,\delta];L^q(\R^d))} \le C\norm{u_0}_{L^2(\R^d)}, \quad h \in (0,1],
\end{align}
for any $p,q \ge 2$.

Next, we shall prove the following dispersive estimate for the main terms:
\begin{equation}
\begin{aligned}
\label{proof_1}
&\norm{b(hD)a^\pm(x,hD)^*e^{-i(t-s)H}a^\pm(x,hD)b(hD)}_{\L(L^1(\R^d),L^\infty(\R^d))}\\
&\le C|t-s|^{-d/2}
\end{aligned}
\end{equation}
for $0<|t-s|\le \delta$. 
We first consider the outgoing case. Let us fix $N >1$ so large that $N \ge 2d+1$. 
After rescaling $t-s \mapsto (t-s)h$ and choosing $R_0>1$ large enough, 
we apply Theorem \ref{3_theorem_1} with $R=R_0$, Lemma \ref{3_lemma_1} and Theorem \ref{A_theorem_1} with $R=h^{-1}$ to $e^{-i(t-s)hH}a^+(x,hD)$. Then, we can write 
\begin{align*}
&e^{-i(t-s)hH}a^+(x,hD)\\
&=\F_{{\IK}}(S_h^+,b^+_h)e^{i(t-s)h\Delta/2}\F_{{\IK}}(S^+_h,c^+_h)^*
+\F_{{\WKB}}(\Psi_h(t-s),b_h(t-s))\\
&\ \ \ +Q_2^+(t-s,h), 
\end{align*}
where the distribution kernels of main terms satisfy dispersive estimates
\begin{align}
\label{proof_1_2}
|K^+_{{\IK}}(t-s,h,x,y)|+|K_{{\WKB}}(t-s,h,x,y)| \le C|(t-s)h|^{-d/2},
\end{align}
uniformly with respect to $h\in (0,h_0]$, $0 < t-s\ \le \delta h^{-1}$ and $x,y \in \R^d$. Let $A(h,x,y)$ and $B(h,x,y)$ be the distribution kernels of $a(x,hD)^*$ and $b(hD)$, respectively. 
They clearly satisfy
$$\sup_x\int (|A(h,x,y)|+|B(h,x,y)|)dy+\sup_y\int (|A(h,x,y)|+|B(h,x,y)|)dx \le C$$
uniformly in $h \in (0,1]$. 
By using this estimate and \eqref{proof_1_2}, we see that the distribution kernel of $b(hD)a^+(x,hD)^*\left(e^{-i(t-s)hH}a^+(x,hD)-Q_2^+(t-s,h)\right)b(hD)$ satisfies the same dispersive estimates as \eqref{proof_1_2} for $0<t-s \le \delta h^{-1}$. 
On the other hand, $Q_2^+(t-s,h)$ satisfy
$$
\norm{Q_2^+(t-s,h)}_{\L(L^2(\R^d))} \le C_Nh^N,\quad h \in (0,h_0],\ 0\le t-s \le \delta h^{-1}. 
$$
We here recall that $a^+(x,hD)^*$ is uniformly bounded on $L^2(\R^d)$ in $h\in (0,1]$ and $b(hD)$ satisfies 
\begin{align*}
&\norm{b(hD)}_{\L(H^{-s}(R^d),H^s(\R^d)} \\
&\le \norm{\<D\>^s\<hD\>^{-s}}_{\L(L^2(R^d)}\norm{\<hD\>^{s}b(hD)\<hD\>^{s}}_{\L(L^2(R^d)}\norm{\<hD\>^{-s}\<D\>^s}_{\L(L^2(R^d)}\\
&\le C_s h^{-2s}. 
\end{align*}
$b(hD)a^+(x,hD)^*Q_2^+(t-s,h)b(hD)$ hence is a bounded operator in $\L(H^{-s},H^s)$ for some $s>d/2$.  Its distribution kernel $\wtilde{Q}_2^+(t-s,h,x,y)$ thus is uniformly bounded on $\R^{2d}$ with respect to $h\in (0,h_0]$ and $0 \le t-s \le \delta h^{-1}$. 
Therefore,  
$$|\wtilde{Q}_2^+(t-s,h,x,y)| \lesssim 1 \lesssim |(t-s)h|^{-d/2},\quad h\in (0,h_0],\ 0 < t-s \le \delta h^{-1}.$$
The corresponding estimates for the incoming case also hold for $0 \le -(t-s) \le \delta h^{-1}$. Therefore, $b(hD)a^\pm(x,hD)^*e^{-i(t-s)hH}a^\pm(x,hD)b(hD)$ have distribution kernels $K^\pm(t-s,h,x,y)$ satisfying
\begin{align}
\label{proof_2}
\abs{K^\pm(t-s,h,x,y)} \le C|(t-s)h|^{-d/2}
\end{align}
uniformly with respect to $h \in (0,h_0],\ 0 \le \pm(t-s)\le \delta h^{-1}$ and $x,y\in \R^d$, respectively. 

We here use a simple trick due to Bouclet-Tzvetkov \cite[Lemma 4.3.]{Bouclet_Tzvetkov_1}. If we set 
$U^\pm(t,h)=b(hD)a^\pm(x,hD)^*e^{-ithH}a^\pm(x,hD)b(hD)$, 
then $$U^\pm(s-t,h)=U^\pm(t-s,h)^*,$$ and hence $K^\pm(s-t,h,x,y)=\overline{K^\pm(t-s,h,y,x)}$. 
Therefore, the estimates \eqref{proof_2} also hold for $0<\mp(t-s)\le \delta h^{-1}$ and $x,y\in \R^d$. Rescaling $(t-s)h \mapsto t-s$, we obtain the estimate \eqref{proof_1}. 

Finally, since the $\L(L^2)$-boundedness of $a^\pm(x,hD)^*e^{-itH}$ is obvious, \eqref{proof_0}, \eqref{proof_1} and the Keel-Tao theorem \cite{Keel_Tao} imply the desired semi-classical Strichartz estimates:
$$
\sup_{h \in (0,h_0]}\norm{(1-\chi_0)\psi_0(h^2H_0)e^{-itH}u_0}_{L^p([-\delta,\delta];L^q(\R^d))} \le C\norm{u_0}_{L^2(\R^d)}.
$$
By the virtue of Proposition \ref{2_proposition_1}, we complete the proof of Theorem \ref{theorem_1} (i).


\section{Proof of Theorem \ref{theorem_1} (ii)}
\label{compact}
In this section we prove Theorem \ref{theorem_1} (ii). Suppose that $H$ satisfies Assumption \ref{assumption_A} with $\mu=\nu=0$. We first recall the local smoothing effects for Schr\"odinger operators with at most quadratic potentials proved by Doi \cite{Doi}. For  any $s \in \R$, we set 
$
\B^s:=\{ f \in L^2(\R^d) ; \<x\>^s f \in L^2(\R^d), \<D\>^s f \in L^2(\R^d)\}$, 
and define a symbol $e_s$ by
$$
e_s(x,\xi):=(k(x,\xi)+|x|^2+L(s))^{s/2}\in S((1+|x|+|\xi|)^s,g).
$$ We denote by $E_s$ its Weyl quantization:
$$E_sf(x)=\frac{1}{2\pi}\int e^{i(x-y)\cdot\xi}e_s\left(\frac{x+y}{2},\xi\right)f(y)dyd\xi.$$
Here $L(s)>1$ is a large constant depending on $s$. Then, for any $s \in \R$, there exists $L(s)>0$ such that $E_s$ is a homeomorphism from $\B^{r+s}$ to $\B^r$ for all $r \in \R$, and $(E_s)^{-1}$ is still a Weyl quantization of a symbol in $S((1+|x|+|\xi|)^{-s},g)$. 

\begin{lemma}[The local smoothing effects \cite{Doi}]
\label{compact_1}
Suppose that the kinetic energy $k(x,\xi)$ satisfies the non-trapping condition \eqref{non-trapping}. Then, for any $T>0$ and $\sigma>0$, there exists $C_T>0$ such that
\begin{align}
\label{smoothing effect}
\norm{\<x\>^{-1/2-\sigma}E_{1/2}u}_{{L^2([-T,T];L^2(\R^d))}} \le C_T \norm{u_0}_{L^2}, 
\end{align}
where $u=e^{-itH}u_0$.  
\end{lemma}

\begin{remark}Let $\chi \in C^\infty_0(\R^d)$. \eqref{smoothing effect} implies a usual local smoothing effect:
\begin{align}
\label{local smoothing effect}
\norm{\<D\>^{1/2}\chi u}_{L^2([-T,T];L^2(\R^d))} \le C_T\norm{u_0}_{L^2(\R^d)}.
\end{align}
 Indeed, let $\chi_1\in C^\infty_0(\R^d)$ be such that $\chi_1 \equiv 1$ on $\supp \chi$. We split $\<D\>^{1/2}\chi$ as follows:
\begin{align*}
\<D\>^{1/2}\chi
&=\chi_1\<D\>^{1/2}\chi+[\<D\>^{1/2},\chi_1]\chi,\\
\chi_1\<D\>^{1/2}\chi&=\chi_1\<D\>^{1/2}(E_{1/2})^{-1}E_{1/2}\chi\\
&=\chi_1\<D\>^{1/2}(E_{1/2})^{-1}\chi_1E_{1/2}\chi+\chi_1\<D\>^{1/2}(E_{1/2})^{-1}[E_{1/2},\chi_1]\chi. 
\end{align*}
By a standard symbolic calculus, $[\<D\>^{1/2},\chi_1]\chi$, $\chi_1\<D\>^{1/2}(E_{1/2})^{-1}$ and $[E_{1/2},\chi_1]\chi$ are bounded on $L^2(\R^d)$ since $\chi_1$ has a compact support. Therefore, Lemma \ref{compact_1} implies
\begin{align*}
\norm{\<D\>^{1/2}\chi u}_{{L^2([-T,T];L^2(\R^d))}}
&\le C\norm{\chi_1 E_{1/2}\chi u}_{{L^2([-T,T];L^2(\R^d))}}+ C_T\norm{u}_{L^2(\R^d)}\\
&\le C_T \norm{u_0}_{L^2(\R^d)}.
\end{align*}
\end{remark}

\begin{bew}{of Theorem \ref{theorem_1} (ii)} We consider the case when $0 \le t \le T$ only, and the proof for the negative time is similar. We mimic the argument in \cite[Section II.2]{Robbiano_Zuily}. A direct computation yields 
\begin{align*}
(i\partial_t +\Delta)\chi u
&=\Delta \chi u + \chi H u\\
&=\chi_1 (H+\Delta) \chi_1 \chi u + (\chi_1[\chi,H]+[\Delta,\chi_1]\chi)u.
\end{align*}
We define a self-adjoint operator by $\widetilde{H}:=-\Delta+\chi_1 (H+\Delta) \chi_1$, and set $$\wtilde{U}(t):=e^{-it\widetilde{H}},\ F:=(\chi_1[\chi,H]+[\Delta,\chi_1]\chi)u.$$
We here note that if $H_0$ satisfies the non-trapping condition then so does the principal part of $\widetilde{H}$. By the Duhamel formula, we can write
$$\chi u = \wtilde{U}(t)\chi u_0 +\int^t_0 \wtilde{U}(t-s)F(s)ds.$$
Since $\chi_1 (H+\Delta) \chi_1$ is a compactly supported smooth perturbation, it was proved by Staffilani-Tataru \cite{Staffilani_Tataru} that $\wtilde{U}(t)$ is bounded from $L^2(\R^d)$ to $L^2([0,T];H^{1/2}_{loc}(\R^d))$, and that its adjoint
$$\wtilde{U}^*f=\int_0^T U(-s)f(s,\cdot)ds$$
is bounded from $L^2([0,T];H^{-1/2}_{loc}(\R^d))$ to $L^2(\R^d)$. Moreover, $\wtilde{U}(t)$ satisfies Strichartz estimates (for any admissible pair $(p,q)$):
\begin{align*}
\norm{\wtilde{U}(t)v}_{L^p([-T,T];L^q(\R^d))} \le C_T \norm{v}_{L^2}, 
\end{align*}
Therefore, we have 
\begin{align*}
\bigg|\bigg|\int^T_0 \wtilde{U}(t-s)F(s)ds\bigg|\bigg|_{L^p([-T,T];L^q(\R^d))} 
&\le C_T \norm{U^*F}_{L^2(\R^d)}\\
&\le C_T \norm{\<D\>^{-1/2}F}_{L^2([-T,T];L^2(\R^d))}
\end{align*}
since $F$ has a compact support with respect to $x$. The Christ-Kiselev lemma (see \cite{Christ_Kiselev,Smith_Sogge}) then implies
$$
\bigg|\bigg|\int^t_0 \wtilde{U}(t-s)F(s)ds\bigg|\bigg|_{L^p([-T,T];L^q(\R^d))} 
\le C_T \norm{\<D\>^{-1/2}F}_{L^2([-T,T];L^2(\R^d))}, 
$$
provided that $p >2$. We split $F$ as $$F= ([\chi,H]\chi_1+[\Delta,\chi_1]\chi)u+[\chi_1,[\chi,H]]u=:F_1+F_2.$$ Since $[\chi,H]$ is a first order differential operator with bounded coefficients, we see that $[\chi_1,[\chi,H]]$ is bounded on $L^2(\R^d)$, and $\norm{\<D\>^{-1/2}F_2}_{L^2([-T,T];L^2(\R^d))}$ is dominated by $C_T\norm{u_0}_{L^2(\R^d)}$
We now use \eqref{local smoothing effect} and obtain 
\begin{align*}
\norm{\<D\>^{-1/2}F_1}_{L^2([-T,T];L^2(\R^d))}
&\le C\norm{\chi_1 u}_{L^2([-T,T];H^{-1/2}(\R^d))}\\
&\le C \norm{ \<D\>^{1/2}\chi_1u}_{L^2([-T,T];L^2(\R^d))}\\
&\le C_T \norm{u_0}_{L^2}, 
\end{align*}
which completes the proof.
  \end{bew}


\appendix

\section{Proof of Propositon \ref{WKB_2_proposition_1}}
\label{appendix_A}
Assume Assumption A with $\mu=0$, $\nu\ge 0$. We here give the detail of the proof of Propositon \ref{WKB_2_proposition_1}. We first study the corresponding classical mechanics. Consider the Hamilton flow 
$$(X_h(t),\Xi_h(t))=(X_h(t,x,\xi),\Xi_h(t,x,\xi)),\quad h \in (0,1],$$ generated by the semi-classical total energy 
$$
p_h(x,\xi)=k(x,\xi)+h^2V(x),
$$
 \emph{i.e.}, $(X_h(t),\Xi_h(t))$ is the solution to the Hamilton equations
\begin{equation}
\left\{
\begin{aligned}
\nonumber
\dot{X}_{h,j}(t)&=\sum_k a^{jk}(X_h(t))\Xi_{h,k}(t),\\
\dot{\Xi}_{h,j}(t)&=-\frac12 \sum_{k,l}\frac{\partial a^{kl}}{\partial x_j}(X_h(t))\Xi_{h,k}(t) \Xi_{h,l}(t)-h^2\frac{\partial V}{\partial x_j}(X_h(t)),
\end{aligned}
\right.
\end{equation}
with the initial condition $(X_h(0),\Xi_h(0))=(x,\xi)$, where $\dot{f}=\partial_t f$. We first prepare an a priori bound of the flow.

\begin{lemma}						
\label{WKB_1_lemma_1}
For all $h \in (0,1]$, $|t| \lesssim h^{-1}$ and $(x,\xi) \in \R^{2d}$,  
\begin{align*}
|X_h(t)-x| \lesssim \left(|\xi|+h\<x\>^{1-\nu/2}\right)|t|,\quad |\Xi_h(t)| \lesssim |\xi|+h\<x\>^{1-\nu/2}.
\end{align*}
\end{lemma}

\begin{proof}
We consider the case $ t \ge 0$. The proof for the case $t<0$ is analogous. Since the Hamilton flow conserves the total energy, namely 
$$p_h(x,\xi)=p_h(X_h(t),\Xi_h(t))\quad\text{for all}\quad t \in \R,$$
 we have
\begin{align*}
|\Xi_h(t)| 
&\lesssim \sqrt{p_0(X_h(t),\Xi_h(t))}\\
&\lesssim \sqrt{p_h(x,\xi)-h^2V(X_h(t))}\\
& \lesssim |\xi|+h\<x\>^{1-\nu/2}+h\<X_h(t)\>^{1-\nu/2}. 
\end{align*}
Applying the above inequality to the Hamilton equation, we have
\begin{align*}
|\dot{X}^h(t)| 
\lesssim |\Xi_h(t)|\lesssim |\xi|+h\<x\>^{1-\nu/2}+h|X_h(t)-x|.
\end{align*}
Integrating with respect to $t$ and using Gronwoll's inequality, we obtain the assertion since $e^{th} \lesssim |t|$ for $|t| \lesssim h^{-1}$. 
  \end{proof}

Let $J \Subset (0,\infty)$ be an open interval. For sufficiently small $\delta>0$ and for all $0<R \le h^{-1}$, the above lemma implies
\begin{align}
\label{WKB_1_lemma_1_1}
|x|/2 \le |X_h(t,x,\xi)| \le 2|x|
\end{align}
uniformly with respect to $h \in (0,1]$, $|t| \le \delta R$ and $(x,\xi) \in \Omega(R,J)$. 
By using this inequality, we have the following:

\begin{lemma}					
\label{WKB_1_lemma_3}
Let $J,\delta$ be as above. Then, for $h \in (0,1]$, $0<R \le h^{-1}$, $|t| \le \delta R$ and $(x,\xi) \in \Omega(R,J)$, $X_h(t,x,\xi)$ and $\Xi_h(t,x,\xi)$ satisfy 
\begin{equation}
\left\{
\begin{aligned}
\label{WKB_1_lemma_3_1}
|X_h(t)-x| &\le C(1+\delta h\<x\>^{1-\nu})|t|,\\
|\Xi_h(t)-\xi| &\le C(\<x\>^{-1}+h^2\<x\>^{1-\nu})|t|, 
\end{aligned}
\right.
\end{equation}
and, for $|\alpha+\beta| = 1$, 
\begin{equation}
\left\{
\begin{aligned}
\label{WKB_1_lemma_3_2}
|\dderiv{x}{\xi}{\alpha}{\beta}(X_h(t)-x)| 
&\le C_{\alpha\beta}\left(\<x\>^{-|\alpha|}+h^{|\alpha|}\<x\>^{-|\alpha|\nu/2}\right)|t|,\\
|\dderiv{x}{\xi}{\alpha}{\beta}(\Xi_h(t)-\xi)
&\le C_{\alpha\beta}\left(\<x\>^{-1-|\alpha|}+h^{1+|\alpha|}\<x\>^{-(1+|\alpha|)\nu/2}\right)|t|, 
\end{aligned}
\right.
\end{equation}
and, for $|\alpha+\beta| \ge 2$,
\begin{equation}
\left\{
\begin{aligned}
\label{WKB_1_lemma_3_2_2}
|\dderiv{x}{\xi}{\alpha}{\beta}(X_h(t)-x)| 
&\le C_{\alpha\beta}\delta h^{|\alpha|}\<x\>^{-1}R|t|,\\
|\dderiv{x}{\xi}{\alpha}{\beta}(\Xi_h(t)-\xi)| 
&\le C_{\alpha\beta} h^{|\alpha|}\<x\>^{-1}|t|.
\end{aligned}
\right.
\end{equation}
Moreover $C,C_{\alpha\beta}>0$ may be taken uniformly with respect to $R$, $h$ and $t$.
\end{lemma}

\begin{proof}
We only prove the case when $t \ge 0$, the proof for the case $t \le0$ is similar. Applying Lemma \ref{WKB_1_lemma_1} and \eqref{WKB_1_lemma_1_1} to the Hamilton equation, we have
\begin{align*}
|\dot{\Xi}^h(t)| 
&\lesssim \<X_h(t)\>^{-1}|\Xi_h(t)|^2+h^2\<X_h(t)\>^{1-\nu}\\
&\lesssim \<x\>^{-1}(1+h^2\<x\>^{2-\nu})+h^2\<x\>^{1-\nu}\\
&\lesssim \<x\>^{-1}+h^2\<x\>^{1-\nu},\\
|\dot{X}^h(t)| &\lesssim |\Xi_h(t)| \lesssim 1+\delta h\<x\>^{1-\nu},
\end{align*}
 and \eqref{WKB_1_lemma_3_1} follows. 
 
 We next prove \eqref{WKB_1_lemma_3_2}. By differentiating the Hamilton equation with respect to $\partial_x^\alpha\partial_\xi^\beta$, $|\alpha+\beta|=1$, we have 
\begin{align}
\label{WKB_1_lemma_3_3}
\frac{d}{dt}
\left(\begin{matrix}
\dderiv{x}{\xi}{\alpha}{\beta} X_h\\
\dderiv{x}{\xi}{\alpha}{\beta} \Xi_h
\end{matrix}\right)
=
\left(\begin{matrix}
\partial_{x}\partial_\xi p_h(X_h,\Xi_h) & \partial_{\xi}^2p_h(X_h,\Xi_h)\\
-\partial_{x}^2p_h(X_h,\Xi_h) & -\partial_{\xi}\partial_{x}p_h(X_h,\Xi_h)
\end{matrix}\right)
\left(\begin{matrix}
\dderiv{x}{\xi}{\alpha}{\beta} X_h\\
\dderiv{x}{\xi}{\alpha}{\beta} \Xi_h
\end{matrix}\right).
\end{align} 
Define a weight function 
$
w_h(x)=\<x\>^{-1}+h\<x\>^{-\nu/2}.
$ 
A direct computation and \eqref{WKB_1_lemma_3_1} then imply
\begin{align*}
\abs{(\dderiv{x}{\xi}{\alpha}{\beta}p_h)(X_h(t),\Xi_h(t))}
&\le C_{\alpha\beta}w_h(x)^{|\alpha|},\quad
|\alpha+\beta| = 2,\\
\abs{(\dderiv{x}{\xi}{\alpha}{\beta}p_h)(X_h(t),\Xi_h(t))}
&\le C_{\alpha\beta}\<x\>^{2-|\alpha+\beta|}w_h(x)^{|\alpha|-1},\quad
|\alpha+\beta| \ge 3,
\end{align*}
for all $|t| \le \delta R$ and $(x,\xi) \in \Omega(R,J)$, and $\partial_\xi^{\beta} p_h\equiv 0$ on $\R^{2d}$ for $|\beta| \ge 3$. By integrating \eqref{WKB_1_lemma_3_3} with respect to $t$, we have 
\begin{align*}
&w_h(x)|\dderiv{x}{\xi}{\alpha}{\beta} (X_h(t)-x)|+|\dderiv{x}{\xi}{\alpha}{\beta} (\Xi_h(t)-\xi)|\\
&\lesssim \int_0^t \left(w_h(x)\left( w_h(x)|\dderiv{x}{\xi}{\alpha}{\beta} (X_h(t)-x)|+|\dderiv{x}{\xi}{\alpha}{\beta} (\Xi_h(t)-\xi)|\right)+w_h(x)^{1+|\alpha|} \right)d\tau
\end{align*}
Using Gronwoll's inequality, we have \eqref{WKB_1_lemma_3_2} since $|t| \le \delta R$. 
 
For $|\alpha+\beta| \ge 2$, we shall prove the estimate for $\partial_{\xi_1}^2X_h(t)$ only. Proofs for other cases are similar, and for higher derivatives follow from an induction on $|\alpha+\beta|$. By the Hamilton equation and \eqref{WKB_1_lemma_3_2},  we learn
  \begin{align*}
\partial_{\xi_1}^2X_h
=\partial_{x}\partial_\xi p_h(X_h,\Xi_h) \partial_{\xi_1}^2X_h
+\partial_{\xi}^2p_h(X_h,\Xi_h) \partial_{\xi_1}^2\Xi_h+Q(h,x,\xi)
\end{align*}
where $Q(h,x,\xi)$ satisfies
\begin{align*}
Q(h,x,\xi) 
&\le C\sum_{|\alpha+\beta|=3,|\beta|=1,2}
(\partial_x^\alpha\partial_\xi^\beta p)(X_h,\Xi_h)
(\partial_{\xi_1}X_h)^{|\alpha|}(\partial_{\xi_1}\Xi_h)^{|\beta|}\\
&\le C\<x\>^{-1}\sum_{|\alpha|=1,2,3}w_h(x)^{|\alpha|-1}|t|^{|\alpha|}\\
&\le C\delta\<x\>^{-1}R. 
\end{align*}
We similarly obtain
\begin{align*}
\partial_{\xi_1}^2\Xi_h
=-\partial_{x}^2p_h(X_h,\Xi_h) \partial_{\xi_1}^2X_h
-\partial_{\xi}\partial_{x}p_h(X_h,\Xi_h) \partial_{\xi_1}^2\Xi_h
+O(\<x\>^{-1}),
\end{align*}
and these estimates and Gronwoll's inequality imply
\begin{align*}
&(\delta R)^{-1}|\partial_{\xi_1}^2X_h(t)|+|\partial_{\xi_1}^2\Xi_h(t)|\\
&\lesssim \int_0^t w_h(x)\left((\delta R)^{-1}|\partial_{\xi_1}^2X_h(t)|+|\partial_{\xi_1}^2\Xi_h(t)|\right)+\<x\>^{-1}d\tau\\
&\lesssim \<x\>^{-1}|t|
\end{align*}
for $0 \le t \le \delta R$. We hence have the assertion.
   \end{proof}

\begin{remark}				
If $\nu =1$, then Lemma \ref{WKB_1_lemma_3} implies that for any $\alpha,\beta \in \Z_+^d$, there exists $C_{\alpha\beta}$ such that
\begin{align}
\label{WKB_remark_1}
|\dderiv{x}{\xi}{\alpha}{\beta}(X_h(t)-x)| 
&\le C_{\alpha\beta}\delta R^{1-|\alpha|},\quad
|\dderiv{x}{\xi}{\alpha}{\beta}(\Xi_h(t)-\xi)
\le C_{\alpha\beta}\delta R^{-|\alpha|}, 
\end{align}
uniformly with respect to $h \in (0,1]$, $0<R \le h^{-1}$, $|t| \le \delta R$ and $(x,\xi) \in \Omega(R,J)$. 
\end{remark}

\begin{lemma}							
\label{WKB_2_lemma_1}
Suppose that $\nu= 1$ and let $J_1 \Subset J_1' \Subset (0,\infty)$ be open intervals. Then there exists $\delta>0$ small enough such that, for any fixed $|t| \le \delta R$, the map
$$
g_h(t):(x,\xi) \mapsto (X_h(t,x,\xi),\xi)
$$
is a diffeomorphism from $\Omega(R/2,J_1')$ onto its range. Moreover, we have
\begin{align}
\label{WKB_2_lemma_1_1}
\Omega(R,J_1) \subset g^h(t,\Omega(R/2,J_1')),\quad |t| \le \delta R.
\end{align}
\end{lemma}

\begin{proof}
We choose $J_1''$ so that $J_1'\Subset J_1'' \Subset (0,\infty)$. Choosing $\chi \in S(1,g)$ such that 
$$
0 \le \chi \le 1,\ 
\supp \chi \subset \Omega(R/3,J_1''),\ 
\chi \equiv 1\ \text{on}\ \Omega(R/2,J_1'),
$$
we define $X_h^\chi(t,x,\xi):=(1-\chi(x,\xi))x+\chi(x,\xi)X_h(t,x,\xi)$ and set $$g_h^\chi(t,x,\xi)=(X_h^\chi(t,x,\xi),\xi).$$ 
We also define $(z,\xi) \mapsto\tilde{g}_h^\chi(t,z,\xi)$ by 
$$
\tilde{g}_h^\chi(t,z,\xi)
=(\tilde{X}^\chi_h(t,z,\xi),\xi)
:=(X^\chi_h(t,Rz,\xi)/R,\xi) .
$$ 
By \eqref{WKB_remark_1}, there exists $\delta>0$ so small that, for $|t| \le \delta R,\ (z,\xi) \in \R^{2d}$,
$$
|\dderiv{z}{\xi}{\alpha}{\beta}(\tilde{X}^\chi_h(t,z,\xi)-z)| 
\lesssim \delta R^{-|\alpha|},\quad
|\dderiv{z}{\xi}{\alpha}{\beta}(J(\tilde{g}_h^\chi)(t,z,\xi)-\Id)| \le C_{\alpha\beta} \delta<1/2,
$$
where $J(\tilde{g}_h^\chi)$ is the Jacobi matrix with respect to $(z,\xi)$. 
The Hadamard global inverse mapping theorem then shows that $\tilde{g}_h^\chi(t)$ is a diffeomorphism from $\R^{2d}$ onto itself if $|t| \le \delta R$. By definition, $g_h(t)$ is a
 diffeomorphism from $\Omega(R/2,J_1')$ onto its range. 

We next prove \eqref{WKB_2_lemma_1_1}. 
Since $g_h(t)=g^\chi_h(t)$ and $g^\chi_h(t)$ is bijective on $\Omega(R/2,J_1')$, it suffices to check that
$$
\Omega(R,J_1)^c \supset g^\chi_h(t,\Omega(R/2,J_1')^c).
$$
Suppose that $(x,\xi) \in \Omega(R/2,J_1')^c$. If $(x,\xi)  \in \Omega(R/3,J_1'')^c$, then 
$$
g^\chi_h(t,x,\xi)=(x,\xi) \in \Omega(R/3,J_1'')^c \subset \Omega(R,J_1)^c. 
$$
Suppose that $(x,\xi) \in \Omega(R/3,J_1'') \setminus \Omega(R/2,J_1')$. By \eqref{WKB_1_lemma_3_1} and the support property of $\chi$, we have
\begin{align*}
|X_h^\chi(t)|
\le |x|+|\chi(X_h(t)-x)| \le  R/2+C\delta R
\end{align*}
for some $C>0$ independent of $R$ and $h$. Choosing $\delta$ satisfying $1/2+C\delta<1$, we obtain
$g_h^\chi(t,x,\xi) \in \Omega(R,J_1)^c$.
  \end{proof}

Let $\Omega(R,J_1) \ni (x,\xi)\mapsto(Y_h(t,x,\xi),\xi)$ be the inverse of $\Omega(R/2,J_1') \in (x,\xi) \mapsto (X_h(t,x,\xi),\xi)$. 
\begin{lemma}							
\label{WKB_2_lemma_2}
Let $\delta,J_1$ as above and $\nu = 1$. Then, for all $h \in (0,1]$, $0<R \le h^{-1}$, $0<|t| \le \delta R$ and $(x,\xi) \in \Omega(R,J_1)$, we have
\begin{align*}
|\dderiv{x}{\xi}{\alpha}{\beta}(Y_h(t,x,\xi)-x)| 
&\le C_{\alpha\beta}\delta R^{1-|\alpha|},\\
|\dderiv{x}{\xi}{\alpha}{\beta}(\Xi_h(t,Y_h(t,x,\xi))-\xi)
&\le C_{\alpha\beta}\delta R^{-|\alpha|}.
\end{align*}
\end{lemma}

\begin{proof}
We prove the inequalities for $Y_h$ only. Proofs for $\Xi_h(t,Y_h(t,x,\xi),\xi)$ are similar. 
Since $(Y_h(t,x,\xi),\xi) \in \Omega(R/2,J_1')$,  
\begin{align*}
|Y_h(t,x,\xi)-x|
&= |X_h(0,Y_h(t,x,\xi),\xi)-X_h(t,Y_h(t,x,\xi),\xi)|\\
&\le \sup_{(x,\xi)\in\Omega(R/2,J_1')}|X_h(t,x,\xi)-x| \\
&\lesssim \delta R. 
\end{align*}
Next, let $\alpha,\beta \in \Z_+^d$ with $|\alpha+\beta|=1$ and apply $\dderiv{x}{\xi}{\alpha}{\beta}$ to the equality 
$$x=X_h(t,Y_h(t,x,\xi),\xi).$$ 
We then have the following equality
\begin{align}
\label{proof_phase_function2_2}
A(t,Z_h(t))
\dderiv{x}{\xi}{\alpha}{\beta}(Y_h(t,x,\xi)-x)
=\dderiv{y}{\eta}{\alpha}{\beta}(y-X_h(t,y,\eta))
|_{(y,\eta)=Z_h(t)},\
\end{align}
where $Z_h(t,x,\xi)=(Y_h(t,x,\xi),\xi)$ and $A(t,Z)=(\partial_x X_h)(t,Z)$. By \eqref{WKB_1_lemma_3_1} and a similar argument as that in the proof of Lemma \ref{WKB_2_lemma_1}, we learn that $A(Z^h(t))$ is invertible, and that $A(Z^h(t))$ and $A(Z^h(t))^{-1}$ are uniformly bounded with respect to $h \in (,1]$, $|t| \le \delta R$ and $(x,\xi) \in \Omega(R,J_1)$. Therefore, 
\begin{align*}
\abs{\dderiv{x}{\xi}{\alpha}{\beta}(Y_h(t,x,\xi)-x)} 
&\le \sup_{(x,\xi)\in\Omega(R/2,J_1')}\abs{\dderiv{y}{\eta}{\alpha}{\beta}(y-X_h(t,y,\eta))}\\
&\le C_{\alpha\beta}\delta R^{1-|\alpha|}. 
\end{align*}
The proof for higher derivatives is obtained by an induction on $|\alpha+\beta|$, and we omit the details. 
  \end{proof}

\begin{bew}{of Proposition \ref{WKB_2_proposition_1}}
We consider the case when $t \ge0$, and the proof for $t \le0$ is similar. Choosing $J \Subset J_1\Subset (0,\infty)$, we define the action integral $\wtilde{\Psi}_h(t,x,\xi)$ on $[0,\delta R] \times \Omega(R/2,J_1)$ by
$$ 
\wtilde{\Psi}_h(t,x,\xi):=x\cdot\xi+\int_0^tL_h(X_h(s,Y_h(t,x,\xi),\xi),\Xi_h(s,Y_h(t,x,\xi),\xi))ds, 
$$
where $L_h(x,\xi)=\xi\cdot\partial_\xi p_h(x,\xi)-p_h(x,\xi)$ is the Lagrangian associated to $p_h$ and $Y_h$ is defined by the above argument with $R>0$ replaced by $R/2$. The smoothness property of $\wtilde{\Psi}_h$ follows from corresponding properties of $X_h$, $\Xi_h$ and $Y_h$. By the standard Hamilton-Jacobi theory, $\wtilde{\Psi}_h(t,x,\xi)$ solves the Hamilton-Jacobi equation \eqref{WKB_2_proposition_1_1} on $\Omega(R/2,J_1)$ and satisfies
$$
\partial_x\wtilde{\Psi}_h(t,x,\xi)=\Xi_h(t,Y_h(t,x,\xi),\xi),\quad \partial_\xi \wtilde{\Psi}_h(t,x,\xi)=Y_h(t,x,\xi).
$$
In particular, we obtain the following energy conservation law:
$$p_h(x,\partial_x\wtilde{\Psi}_h(t,x,\xi))=p_h(Y_h(t,x,\xi),\xi).$$ 
This energy conservation and Lemma \ref{WKB_2_lemma_2} imply
\begin{align*}
&|p_h(\partial_x\wtilde{\Psi}_h(t,x,\xi)-p_h(x,\xi)|\\
&\le
|Y_h(t,x,\xi)-x)|\int_0^1|\partial_x p_h(\lambda x+(1-\lambda)Y_h(t,x,\xi),\xi)|d\lambda\\
&\le C\delta R(\<x\>^{-1}+h^2) \\
&\le C\delta.
\end{align*}
By using Lemma \ref{WKB_2_lemma_2}, we also obtain 
\begin{align*}
|\dderiv{x}{\xi}{\alpha}{\beta}(p_h(x,\partial_x\wtilde{\Psi}_h(t,x,\xi))-p_h(x,\xi))|
&\le C_{\alpha\beta}
\delta R^{|\alpha|},\quad \alpha,\beta \in \Z^d_+.
\end{align*}
Therefore, 
$$
\abs{\dderiv{x}{\xi}{\alpha}{\beta}\left(\wtilde{\Psi}_h(t,x,\xi)-x\cdot\xi+tp_h(x,\xi)\right)}
\le C_{\alpha\beta} \delta R^{|\alpha|}|t|. 
$$
Choose $\chi \in S(1,g)$ so that $$
0 \le \chi\le 1,\ 
\chi \equiv 1\ \text{on}\ \Omega(R,J)\ \text{and}\ 
\supp \chi \subset \Omega(R/2,J_1),$$
 and define $$\Psi_h(t,x,\xi):=x\cdot\xi-tp_h(x,\xi)+\chi(x,\xi)(\wtilde{\Psi}_h(t,x,\xi)-x\cdot\xi+tp_h(x,\xi)).$$ Clearly, $\Psi_h(t,x,\xi)$ satisfies the statement of Proposition \ref{WKB_2_proposition_1}. 
  \end{bew}


\end{document}